\documentclass[12pt,a4paper]{article}

\usepackage{fullpage}
\usepackage[utf8]{inputenc}
\usepackage{amssymb}
\usepackage{amsmath}
\usepackage{amsthm}
\usepackage[round]{natbib}
\usepackage{url}
\usepackage{graphicx}
\usepackage{multirow}
\usepackage{setspace}
\usepackage{enumerate}
\usepackage{color}
\usepackage[colorlinks=true,linkcolor=blue,citecolor=blue,pdfborder={0 0 0}]{hyperref}
\usepackage{booktabs}
\usepackage{bm}

\newcommand{\eps}{\varepsilon}
\newcommand{\N}{\mathbb{N}}
\newcommand{\R}{\mathbb{R}}

\newcommand{\dd}{\mathrm{d}}

\newcommand{\FF}{\mathcal{F}}

\newcommand{\D}{\mathbb{D}}
\newcommand{\B}{\mathbb{B}}
\newcommand{\Hb}{\mathbb{H}}
\newcommand{\U}{\mathbb{U}}
\newcommand{\V}{\mathbb{V}}
\newcommand{\W}{\mathbb{W}}
\newcommand{\K}{\mathbb{K}}
\newcommand{\Lb}{\mathbb{L}}
\newcommand{\Z}{\mathbb{Z}}
\newcommand{\Ex}{\mathrm{E}}
\newcommand{\Var}{\mathrm{Var}}
\newcommand{\Cov}{\mathrm{Cov}}
\newcommand{\MSE}{\mathrm{MSE}}

\newcommand{\1}{\mathbf{1}}

\newcommand{\ip}[1]{\lfloor #1 \rfloor}

\renewcommand{\vec}{\bm}

\renewcommand{\Pr}{\mathrm{P}}
\newcommand{\p}{\overset{\Pr}{\to}}
\newcommand{\as}{\overset{\mathrm{a.s.}}{\longrightarrow}}

\theoremstyle{plain}
\newtheorem{prop}{Proposition}[section]
\newtheorem{defn}[prop]{Definition}

\newtheorem{lem}[prop]{Lemma}

\numberwithin{equation}{section}

\title{Dependent multiplier bootstraps for non-degenerate $U$-statistics under mixing conditions with applications} 

\author{Axel B\"ucher\,\footnote{Ruhr-Universit\"at Bochum,
Fakult\"at f\"ur Mathematik, 
Universit\"atsstr.~150, 44780 Bochum, Germany. 
{E-mail:} \texttt{axel.buecher@rub.de}}\;~ and Ivan Kojadinovic\,\footnote{Universit\'e de Pau et des Pays de l'Adour, Laboratoire de math\'ematiques et applications, UMR CNRS 5142, B.P. 1155, 64013 Pau Cedex, France.
{E-mail:} \texttt{ivan.kojadinovic@univ-pau.fr}.}
}

\begin{document}
\maketitle

\begin{abstract}
The asymptotic validity of a resampling method for two sequential processes constructed from non-degenerate $U$-statistics is established under mixing conditions. The resampling schemes, referred to as {\em dependent multiplier bootstraps}, result from an adaptation of the seminal approach of \cite{GomHor02} to mixing sequences. The proofs exploit recent results of \cite{DehWen10b} on degenerate $U$-statistics. A data-driven procedure for estimating a key bandwidth parameter involved in the resampling schemes is also suggested, making the use of the studied dependent multiplier bootstraps fully automatic. The derived results are applied to the construction of confidence intervals and to test for change-point detection. For such applications, Monte Carlo experiments suggest that the use of the proposed resampling approaches can have advantages over that of estimated asymptotic distributions. 

\medskip

\noindent {\it Keywords:} alpha and beta mixing; change-point detection; functional multiplier central limit theorem; lag window estimator; sequential processes.

\end{abstract}


\section{Introduction}

The asymptotic analysis of many well-known estimators and tests can be carried out using the theory of $U$-statistics. Common examples of estimators are the empirical variance, Gini's mean difference or Kendall's rank correlation coefficient, while a classical test based on a $U$-statistic is Wilcoxon's signed rank test for the hypothesis of location at zero (see, e.g., \citealp{van98}, Example 12.4).
Throughout this work, we focus on the important special case of $U$-statistics of order 2 based on stationary, short-range dependent observations. More precisely,  
let $d \geq 1$ be an integer and let $h: \R^d \times \R^d \to \R$ be a symmetric, measurable function. Given a stretch of observations $\vec X_1,\dots,\vec X_n$ drawn from a stationary, $\R^d$-valued sequence $(\vec X_i)_{i \in \N}$, 
\begin{equation}
\label{eq:Uh1n}
U_{h,1:n} = \frac{1}{\binom{n}{2}} \sum_{1 \leq i < j \leq n} h(\vec X_i,\vec X_j)
\end{equation}
is called {\em $U$-statistic of order 2 with kernel $h$}.

To analyze the asymptotics of such $U$-statistics, \cite{Hoe48b} introduced the decomposition
\begin{equation}
\label{eq:Hoeff_decomp}
U_{h,1:n} = \theta + \frac{2}{n} \sum_{i=1}^n h_1(\vec X_i) + U_{h_2,1:n},
\end{equation}
where, with $\vec X$ and $\vec X'$ denoting independent random vectors that have the same distribution as~$\vec X_1$,
\begin{align}
\label{eq:theta}
\theta &= \Ex\{h(\vec X, \vec X' ) \}, \\
\label{eq:h1}
h_1(\vec x) &= \Ex\{ h(\vec x, \vec X) \} - \theta, \qquad \vec x \in \R^d, \\
\label{eq:h2}
h_2(\vec x_1,\vec x_2) &= h(\vec x_1,\vec x_2) - h_1(\vec x_1) - h_1(\vec x_2) - \theta, \qquad \vec x_1, \vec x_2 \in \R^d,
\end{align}
provided all integrals exist. A simple calculation shows that $h_2$ is a {\em degenerate} kernel in the sense that $\Ex\{h_2(\vec x,\vec X)\} = 0$ for all $\vec x \in \R^d$. If $\Var\{ h_1(\vec X)\}=0$,  it follows from~\eqref{eq:Hoeff_decomp} that the asymptototic behavior of $U_{h,1:n}$ is determined by that of $U_{h_2,1:n}$, whence $U_{h,1:n}$ is called a \textit{degenerate $U$-statistic}. 
In the rest of this work, we restrict our attention to $U$-statistics such that $\Var\{ h_1(\vec X)\} > 0$ and refer to them as {\em non-degenerate $U$-statistics}. 

For any $x \geq 0$, let $\ip{x}$ be the greatest integer smaller or equal than $x$, and, for any $(s,t) \in \Delta= \{ (s,t) \in [0,1]^2: s \le t \}$, let $\lambda_n(s,t) = (\ip{nt}-\ip{ns} )/ n$. Also, let $\ell^\infty([0,1])$ be the space of all bounded real-valued functions on $[0,1]$ equipped with the uniform metric. The main theoretical aim of this work is to establish, under suitable moment and mixing conditions, the asymptotic validity of two {\em dependent multiplier bootstraps} for the stochastic process $\U_n \in \ell^\infty([0,1])$ defined by
\begin{equation}
\label{eq:Un}
\U_n(s) = 
\sqrt{n} \lambda_n(0,s) ( U_{h, 1:\ip{ns}} - \theta ) \qquad
\mbox{if } s \in [2/n,1],
\end{equation}
and $\U_n(s) = 0$ otherwise, and for the related process $\D_n \in \ell^\infty([0,1])$ defined by
\begin{equation}
\label{eq:Dn}
\D_n (s)  = 
 \sqrt n \lambda_n(0,s) \lambda_n(s,1) ( U_{h, 1:\ip{ns} } - U_{h, \ip{ns}+1:n } )  \qquad \mbox{if } s \in [2/n,1-2/n],
\end{equation}
and $\D_n(s)=0$ otherwise. The latter process is of particular importance for change point analysis; see Section~\ref{subsec:cp} below. 

Multiplier bootstraps, also frequently referred to as {\em wild} or {\em weighted} bootstraps, were used in a wide variety of settings. For the arithmetic mean, such resampling schemes were investigated among others by \cite{BarBer86} for independent observations and by \cite{Sha10} for weakly dependent observations. The latter author in particular showed that the {\em dependent} multiplier bootstrap shares the same favorable asymptotic properties as the {\em tapered block bootstrap} of \cite{PapPol01}: The mean squared error of the corresponding variance estimator can be of order $O(n^{-4/5})$, which compares favorably to the best rate of $O(n^{-2/3})$ achieved by all other time series bootstraps such as the {\em block bootstrap} of \cite{Kun89}, the {\em circular bootstrap} of \cite{PolRom92} or the {\em stationary bootstrap} of \cite{PolRom94}. Another advantage of multiplier bootstraps is that they can often be implemented in a computationally efficient way \citep[see, e.g.,][or Section~\ref{sec:confint_theta} below]{KojYanHol11}. For general empirical processes based on independent observations, key theoretical results on the multiplier bootstrap are given in \citet[Chapters 2.9 and 3.6]{vanWel96}, while for standard empirical processes based on weakly dependent observations, a seminal contribution is \citet[Section 3.3]{Buh93} which was recently revisited by \cite{BucKoj14}. Mutiplier bootstraps for degenerate $U$-statistics were for instance studied by \cite{DehMik94} in the case of independent observations and recently by \cite{LeuNeu13} in the case of weakly dependent data. The case of non-degenerate $U$-statistics based on independent observations was investigated by \cite{Jan94} and \cite{WanJin04}, among others. 

For non-degenerate $U$-statistics based on weakly dependent observations, the only study of the asymptotic validity of a resampling scheme seems to be due to \cite{DehWen10a} who investigated a {\em circular block bootstrap} for the statistic $\U_n(1)$ with $\U_n$ defined in~\eqref{eq:Un}. The dependent multiplier bootstraps for the process $\U_n$ proposed in Section~\ref{subsec:unboot} of this work are thus (sequentially extended) alternatives to the latter approach. Our proofs of their asymptotic validity exploit recent key results due to \cite{DehWen10a} and \cite{DehWen10b} concerning the degenerate $U$-statistic $U_{h_2,1;n}$ appearing in Hoeffding's decomposition~\eqref{eq:Hoeff_decomp}. 

We apply similar arguments to prove the asymptotic validity of related dependent multiplier bootstraps for the process $\D_n$ in~\eqref{eq:Dn}. The latter process is a key ingredient in a large class of tests for change-point detection \citep[see, e.g.,][]{GomHor02,HorHus05}, a typical test statistic being
\begin{equation}
\label{eq:Sn}
S_n = \max_{2 \leq k \leq n-2} | \D_n (k/n) | = \sup_{s \in [0,1]} | \D_n(s) |.
\end{equation}
Particular choices for the kernel $h$ lead to, for instance, tests for detecting changes in the variance, Gini's mean difference or Kendall's tau \citep[see][and references therein for more details on this last test]{DehVogWenWie14}. The dependent multiplier bootstraps for $\D_n$ investigated in this work can actually be regarded as an extension of the seminal multiplier bootstrap results of \cite{GomHor02} from independent to weakly dependent observations. As a consequence of this extension, under suitable moment and mixing conditions, the tests for change-point dectection based on $S_n$ could be carried out using resampling instead of relying on the fact that, under the null, $\D_n$ converges weakly to a Brownian bridge depending on an unknown long-run variance parameter that needs to be estimated. Our Monte Carlo experiments indicate that the use of resampling instead of the estimated asymptotic null distribution of $S_n$ can lead to better behaved tests. 

The remaining parts of this article are organized as follows. Dependent multiplier bootstrap results for the process $\U_n$ defined in~\eqref{eq:Un} are given in Section~\ref{sec:unboot}. In addition to asymptotic validity results, a procedure for estimating a key bandwidth parameter (playing a role somehow analoguous to the {\em block length} in the {\em block bootstrap}) is proposed. A straightforward application to the construction of confidence intervals concludes the section and illustrates possible advantages of the use of the proposed resampling schemes. Section~\ref{sec:multDn} provides asymptotic validity results for the related bootstrap procedures for the process $\D_n$ defined in~\eqref{eq:Dn}, and discusses applications to change-point detection. Monte Carlo experiments are carried out for a specific test for change-point detection based on the statistic $S_n$ in~\eqref{eq:Sn} and suggest that, in this case, the use of resampling may be preferable to that of the estimated asymptotic null distribution when computing an approximate p-value for~$S_n$.  Section~\ref{sec:con} concludes. 

All proofs are deferred to a sequence of appendices and the studied tests for change-point detection are implemented in the package {\tt npcp} for the \textsf{R} statistical system \citep{Rsystem}.

\section{Two dependent multiplier bootstraps for $\U_n$}  \label{sec:unboot}

\subsection{Additional definitions}

For the sake of completeness, let us first recall the notions of {\em strongly mixing sequence} and {\em absolutely regular sequence}. For a sequence of $d$-dimensional random vectors $(\vec Y_i)_{i \in \N}$, the $\sigma$-field generated by $(\vec Y_i)_{a \leq i \leq b}$, $a, b \in \N \cup \{+\infty \}$, is denoted by $\FF_a^b$. The strong mixing coefficients corresponding to the sequence $(\vec Y_i)_{i \in \N}$ are then defined by $\alpha_0 = 1/2$ and
$$
\alpha_r = \sup_{p \in \N} \sup_{A \in \FF_{0}^p,B\in \FF_{p+r}^{+\infty}} | \Pr(A \cap B) - \Pr(A) \Pr(B) |, \qquad r \in \N, r > 0.
$$
The sequence $(\vec Y_i)_{i \in \N}$ is said to be {\em strongly mixing} if $\alpha_r \to 0$ as $r \to \infty$. The absolute regularity coefficients corresponding to the sequence $(\vec Y_i)_{i \in \N}$ are defined by
$$
\beta_r = \sup_{p \in \N} \Ex \sup_{A \in \FF_{p+r}^\infty} | \Pr(A \mid \FF_{0}^p ) - \Pr(A) |, \qquad r \in \N, r > 0.
$$
The sequence $(\vec Y_i)_{i \in \N}$ is said to be {\em absolutely regular} if $\beta_r \to 0$ as $r \to \infty$. As $\alpha_r \leq \beta_r$, absolute regularity implies strong mixing. 

To establish the desired theoretical results, we rely on key results of \cite{DehWen10a} and \cite{DehWen10b} on the degenerate part of Hoeffding's decomposition~\eqref{eq:Hoeff_decomp} of a $U$-statistic. The latter require that the kernel $h$ satisfies certain moment conditions.

\begin{defn}
Given a strictly stationary sequence $(\vec X_i)_{i \in \N}$, a kernel $h$ is said to have uniform $\gamma$-moments, $\gamma > 0$, if there exists $B > 0$ such that
$$
\Ex \{ |h(\vec X,\vec X')|^\gamma \} \leq B
$$
for any random vector $(\vec X, \vec X')$ in $\R^{2d}$ with probability distribution in $\{ \Pr^{(\vec X_1, \vec X_k)} : k \in \N \} \cup \{ \Pr^{\vec X_1} \otimes \Pr^{\vec X_1} \}$.
\end{defn}

The following continuity conditions are also needed. The first one is due to \cite{DehWen10a}, the second one to \cite{DenKel86} \citep[see also][Definition 1.7]{DehWen10b}.

\begin{defn}
(a) A kernel $h$ is called $\Pr$-Lipschitz-continuous with constant $L > 0$ if
\[
\Ex \{ |h(\vec X, \vec Y)  - h(\vec X', \vec Y)| \1( \| \vec X - \vec X' \| \le \eps ) \}  \le L \eps
\]
for any $\eps > 0$ and any random vector $(\vec X, \vec X', \vec Y)$ in $\R^{3d}$ such that the probability distributions of $(\vec X, \vec Y)$ and $(\vec X',\vec Y)$ are in $\{ \Pr^{(\vec X_1, \vec X_k)} : k \in \N \} \cup \{ \Pr^{\vec X_1} \otimes \Pr^{\vec X_1} \}$. 

\smallskip
\noindent
(b) A kernel $h$ is said to satisfy the variation condition if there exists a constant $L > 0$  such that,  for any $\eps > 0$ and any independent random vectors $\vec X$ and $\vec X'$ with probability distribution $\Pr^{\vec X_1}$,
\[
\Ex \left\{ \sup_{ \|(\vec y, \vec y') - (\vec X, \vec X')\| \leq \eps \atop \|(\vec z, \vec z') - (\vec X, \vec X')\| \leq \eps } |h(\vec y, \vec y')  - h(\vec z, \vec z')|  \right\}\le L \eps.
\]
\end{defn}

Examples of kernels satisfying the $\Pr$-Lipschitz-continuity condition or the variation condition are given for instance in \cite{DehWen10b}; see also Section~\ref{sec:confint_theta}.

\subsection{Weak convergence of $\U_n$ under mixing}

Before presenting the proposed resampling schemes and stating consistency results, we study the asymptotics of $\U_n$ in~\eqref{eq:Un}. The following proposition is a rather immediate consequence of Theorem~2 of \cite{OodYos72} and Theorem~1 of \cite{DehWen10b}. Its proof is given in Appendix~\ref{proof:prop:weakUn}.

\begin{prop}[Asymptotics of $\U_n$]
\label{prop:weakUn}
Assume that $\vec X_1,\dots,\vec X_n$ is drawn from a strictly stationary sequence $(\vec X_i)_{i \in \N}$ and that $h$ has uniform $(2+\delta)$-moments for some $\delta > 0$. Furthermore, suppose that one of the following two conditions holds:
\begin{enumerate}[(i)]
\item $(\vec X_i)_{i \in \N}$ is absolutely regular and $\beta_r = O(r^{-b})$, $b > (2+\delta)/\delta$,
\item $(\vec X_i)_{i \in \N}$ is strongly mixing, $\Ex ( \| \vec X_1 \|^\gamma ) < \infty$ for some $\gamma > 0$, $h$ satisfies the $\Pr$-Lipschitz continuity or variation condition and $\alpha_r = O(r^{-b})$, $b > \max\{(3\gamma\delta + \delta + 5\gamma + 2)/(2\gamma\delta),(2+\delta)/\delta\}$.
\end{enumerate}
Then, 
\begin{equation}
\label{eq:asym_equiv_Un}
\sup_{s \in [0,1]} \bigg| \U_n(s) - \frac{2}{\sqrt{n}} \sum_{i=1}^{\ip{ns}} h_1(\vec X_i) \bigg| = o_\Pr(1)
\end{equation}
and
\begin{equation}
\label{eq:sigma_h1}
\sigma_{h_1}^2 = \Ex [ \{h_1(\vec X_1)\}^2 ] + 2 \sum_{i=2}^\infty \Ex\{h_1(\vec X_1)h_1(\vec X_i) \} < \infty.
\end{equation}
Consequently, $\U_n \leadsto \U = 2 \sigma_{h_1} \B$ in $\ell^\infty([0,1])$, where~`$\leadsto$' denotes weak convergence in the sense of Definition~1.3.3 in \cite{vanWel96} and $\B$ is a standard Brownian motion.
\end{prop}

Interestingly enough, the sufficient mixing conditions become significantly simpler if $h$ is a bounded kernel. Under such a restriction, the above result was established for {\em $\Pr$-near epoch dependent} sequences by \citet[Theorem B.1]{DehVogWenWie14}.

\subsection{Dependent multiplier bootstraps} \label{subsec:unboot}

The proposed dependent multiplier bootstraps for $\U_n$ rely on the notion of {\em dependent multiplier sequence} due to \citet[Section 3.3]{Buh93} \citep[see also][]{BucKoj14}. 

\begin{defn}
\label{defn:mult_seq}
A sequence of random variables $(\xi_{i,n})_{i \in \N}$ is said to be a {\em dependent multiplier sequence} if:
\begin{enumerate}[({M}1)]
\item 
The sequence $(\xi_{i,n})_{i \in \N}$ is strictly stationary with $\Ex(\xi_{1,n}) = 0$, $\Ex(\xi_{1,n}^2) = 1$ and $\sup_{n \geq 1} \Ex(|\xi_{1,n}|^\nu) < \infty$ for all $\nu \geq 1$, and is independent of the available sample $\vec X_1,\dots,\vec X_n$.
\item 
There exists a sequence $\ell_n \to \infty$ of strictly positive constants such that $\ell_n = o(n)$ and the sequence $(\xi_{i,n})_{i \in \N}$ is $\ell_n$-dependent, i.e., $\xi_{i,n}$ is independent of $\xi_{i+h,n}$ for all $h > \ell_n$ and $i \in \N$. 
\item 
There exists a function $\varphi:\R \to [0,1]$, symmetric around 0, continuous at $0$, satisfying $\varphi(0)=1$ and $\varphi(x)=0$ for all $|x| > 1$ such that $\Ex(\xi_{1,n} \xi_{1+h,n}) = \varphi(h/\ell_n)$ for all $h \in \N$.
\end{enumerate}
\end{defn}

Let $M$ be a large integer and let $(\xi_{i,n}^{(1)})_{i \in \N},\dots,(\xi_{i,n}^{(M)})_{i \in \N}$ be $M$ independent copies of the same dependent multiplier sequence. Then, for any $m \in \{1,\dots,M\}$ and $s \in [0,1]$, let
\begin{equation}
\label{eq:H_Hm}
\Hb_n (s) = \frac{2}{\sqrt{n}} \sum_{i=1}^{\ip{ns}} h_1(\vec X_i), \qquad \Hb_n^{(m)} (s) = \frac{2}{\sqrt{n}} \sum_{i=1}^{\ip{ns}} \xi_{i,n}^{(m)} h_1(\vec X_i),
\end{equation}
where $h_1$ is defined in~\eqref{eq:h1}. The dependent multiplier central limit theorem stated in Proposition~\ref{prop:func_mult} then implies that, under suitable moment and mixing conditions, the processes $\Hb_n, \Hb_n^{(1)}, \dots, \Hb_n^{(M)}$ jointly converge weakly to independent copies of the same limit, suggesting to interpret $\Hb_n^{(1)}, \dots, \Hb_n^{(M)}$ as bootstrap replicates of $\Hb_n$. 

To provide some more insight on the latter statement, and before addressing the fact that the sample $h_1(\vec X_1),\dots,h_1(\vec X_n)$ is not necessarily observable, let us for a brief moment fix $s$ to 1. With the notation $\overline{h_1} = n^{-1} \sum_{i=1}^n h_1(\vec X_i)$, $\Hb_n(1)$ and $\Hb_n^{\scriptscriptstyle (m)}(1)$ can be rewritten as 
 $$
 \Hb_n (1) = 2 \sqrt{n} \left[ \overline{h_1} - \Ex\{h_1(\vec X_1)\} \right], \qquad
 \Hb_n^{(m)} (1) = 2 \sqrt{n} \left[ \frac{1}{n} \sum_{i=1}^n  (\xi_{i,n}^{(m)} +1 ) h_1(\vec X_i)  - \overline{h_1} \right],
 $$
 respectively, suggesting that the $m$th bootstrap sample is $(\xi_{i,n}^{\scriptscriptstyle  (m)} +1 )h_1(X_i)$, $i \in \{1,\dots,n\}$. In the case of the block bootstrap of \cite{Kun89} based on randomly selecting $k$ potentially overlapping blocks of length $\ell_n$ (assume for simplicity that $k=n/\ell_n \in \N$), the $m$th bootstrap sample can be written as $W_{i,n}^{\scriptscriptstyle  (m)} h_1(X_i)$, $i \in \{1,\dots,n\}$, where $W_{i,n}^{\scriptscriptstyle  (m)}$ is the number of blocks that contain $h_1(X_i)$. Proceeding for instance as in \citet[Section 3.3]{Buh93}, it can be verified that $(W_{i,n}^{\scriptscriptstyle   (m)}-1)$, $i \in \{1,\dots,n\}$, can almost be regarded as a portion of a dependent multiplier sequence constructed by taking $\varphi$ in Definition~\ref{defn:mult_seq} to be the triangular (Bartlett) kernel. \citet{Buh93} \citep[see also][]{Sha10,PapPol01} then observed that smoother kernels for $\varphi$ would reduce the bias of the estimator of the underlying long-run variance, thereby improving the order of accuracy of the corresponding mean squared error (see~\eqref{eq:bias_var} and \eqref{eq:MSE} in the next section).
 
Because of~\eqref{eq:asym_equiv_Un}, the multiplier processes $\Hb_n^{(m)}$ in \eqref{eq:H_Hm} can actually be regarded as bootstrap replicates of $\U_n$ as well. They are however not necessarily computable as, depending on the choice of $h$, the sample $h_1(\vec X_1),\dots,h_1(\vec X_n)$ is not necessarily observable. Starting from~\eqref{eq:h1} and given integers $1 \leq k \leq l \leq n$, it is natural to estimate the sample $h_1(\vec X_k),\dots,h_1(\vec X_l)$ by the pseudo-observations $\hat h_{1,k:l}(\vec X_k),\dots, \hat h_{1,k:l}(\vec X_l)$, where
\begin{equation}
\label{eq:hath1kl}
\hat h_{1,k:l}(\vec X_i) = \frac{1}{l-k} \sum_{j=k \atop j \neq i}^l h(\vec X_i, \vec X_j) - U_{h,k:l}, \qquad i \in \{k,\dots,l\},
\end{equation}
with the convention that $\hat h_{1,k:l} = 0$ if $k=l$. Fix $m \in \{1,\dots,M\}$. We then consider the two following computable versions of $\Hb_n^{(m)}$ in~\eqref{eq:H_Hm} defined, for any $s \in [0,1]$, as
\begin{equation}
\label{eq:hatUnm}
\hat \U_n^{(m)} (s) = \frac{2}{\sqrt{n}} \sum_{i=1}^{\ip{ns}} \xi_{i,n}^{(m)} \hat h_{1,1:n} (\vec X_i), \qquad s \in [0,1],
\end{equation}
and
\begin{equation}
\label{eq:checkUnm}
\check \U_n^{(m)} (s) = \frac{2}{\sqrt{n}} \sum_{i=1}^{\ip{ns}} \xi_{i,n}^{(m)} \hat h_{1,1:\ip{ns}} (\vec X_i), \qquad s \in [0,1],
\end{equation} 
respectively. The two processes above are to be interpreted as bootstrap replicates of the process $\U_n$ defined in~\eqref{eq:Un}. The process $\check \U_n^{(m)}$ was considered in the case of independent observations in the seminal work of \cite{GomHor02}, while the process $\hat \U_n^{(m)}$ is a variation of the latter that uses all the available observations to estimate $h_1(\vec X_1),\dots,h_1(\vec X_{\ip{ns}})$. In the related partial-sum setting considered in \cite{BucKojRohSeg14}, the ``check'' approach {\em à la}~\eqref{eq:checkUnm} led to better finite-sample performance, while the ``hat'' approach {\em à la}~\eqref{eq:hatUnm} was found superior in \cite{HolKojQue13}. In the setting under consideration, the quality of the bootstrap approximation might be affected by the kernel $h$, which prompted us to study both approaches theoretically.  The following result is proved in Appendix~\ref{proof:prop:multUn}.

\begin{prop}[Two dependent multiplier bootstraps for $\U_n$]
\label{prop:multUn}
Assume that $\vec X_1,\dots,\vec X_n$ is drawn from a strictly stationary sequence $(\vec X_i)_{i \in \N}$ and that $h$ has uniform $(4+\delta)$-moments for some $\delta > 0$. Also,  let $(\xi_{i,n}^{(1)})_{i \in \N}$,\dots,$(\xi_{i,n}^{(M)})_{i \in \N}$ be independent copies of the same dependent multiplier sequence satisfying (M1)--(M3) in Definition~\ref{defn:mult_seq} such that $\ell_n = O(n^{1/2 - \eps})$ for some $1/(6+2\delta) < \eps < 1/2$. Furthermore, suppose that one of the following two conditions holds:
\begin{enumerate}[(i)]
\item $(\vec X_i)_{i \in \N}$ is absolutely regular with $\beta_r = O(r^{-b})$, $b > 2 (4+\delta)/\delta$,
\item $(\vec X_i)_{i \in \N}$ is strongly mixing, $\Ex ( \| \vec X_1 \|^\gamma ) < \infty$ for some $\gamma > 0$, $h$ satisfies the $\Pr$-Lipschitz continuity or variation condition and $\alpha_r = O(r^{-b})$, $b > \max\{ (3\gamma\delta + \delta + 5\gamma + 2)/(\gamma\delta), 2 (4+\delta)/\delta \}$. 
\end{enumerate}
Then, for any $m \in \{1,\dots,M\}$,
\begin{align}
\label{eq:asym_equiv_hatUnm}
\sup_{s \in [0,1]} \bigg| \hat \U_n^{(m)}(s) - \frac{2}{\sqrt{n}} \sum_{i=1}^{\ip{ns}} \xi_{i,n}^{(m)} h_1(\vec X_i) \bigg| &= o_\Pr(1), \\ 
\label{eq:asym_equiv_checkUnm}
\sup_{s \in [0,1]} \bigg| \check \U_n^{(m)}(s) - \frac{2}{\sqrt{n}} \sum_{i=1}^{\ip{ns}} \xi_{i,n}^{(m)} h_1(\vec X_i) \bigg| &= o_\Pr(1),
\end{align}
and
\begin{align*}
\left(\U_n,\hat \U_n^{(1)},\dots,\hat \U_n^{(M)},\check \U_n^{(1)},\dots,\check \U_n^{(M)} \right) &\leadsto \left(\U,\U^{(1)},\dots,\U^{(M)},\U^{(1)},\dots,\U^{(M)} \right)
\end{align*}
in $\{\ell^\infty([0,1])\}^{2M+1}$,  where $\U$ is the weak limit of $\U_n$ given in Proposition~\ref{prop:weakUn}, and $\U^{(1)},\dots,\U^{(M)}$ are independent copies of $\U$.
\end{prop} 

Notice that most results establishing the asymptotic validity of resampling procedures involve weak convergence of conditional laws. Unlike such results, Proposition~\ref{prop:multUn} above is of an unconditional nature. As explained in \citet[Remark 2.3]{BucKoj14} and as shall be discussed further in the applications of Sections~\ref{sec:confint_theta} and~\ref{sec:multDn} below, the adopted unconditional approach leads to meaningful validity conclusions in most, if not all, statistical contexts of practical interest.

\subsection{Estimation of the bandwidth parameter $\ell_n$}
\label{sec:estim-bandw-param}

From a practical perspective, the use of either of the two dependent multiplier bootstraps studied in the previous section requires the choice of the bandwidth parameter~$\ell_n$ appearing in the definition of dependent multiplier sequences (see Definition~\ref{defn:mult_seq}). As mentioned in \cite{BucKoj14}, since $\ell_n$ plays a role somehow analogous to that of the block length in the block bootstrap, its value is expected to have a crucial influence on the finite-sample performance of the dependent multiplier bootstraps. The aim of this section is to propose an estimator of $\ell_n$ in the spirit of that investigated in \cite{PapPol01} and \cite{PolWhi04}, among others, for other resampling schemes. 

From~\eqref{eq:asym_equiv_hatUnm} and~\eqref{eq:asym_equiv_checkUnm}, we see that the two dependent multiplier bootstraps under consideration are asymptotically equivalent to a dependent multiplier bootstrap for the mean (multiplied by 2) of the typically unobservable sequence $h_1(\vec X_1),\dots,h_1(\vec X_n)$. Analogously to \cite{PapPol01} \citep[see also][]{PolWhi04,PatPolWhi09}, the idea is then to estimate $\ell_n$ as the value that minimizes asymptotically the mean square error of
$$
\sigma_{h_1,n}^2 = \Var_\xi \left\{ \frac{1}{\sqrt{n}} \sum_{i=1}^n \xi_{i,n} h_1(\vec X_i) \right\}, 
$$
where $\Var_\xi$ denotes the variance conditional on the data and $(\xi_{i,n})_{i \in \N}$ is a dependent multiplier sequence. Interestingly enough, it is easy to verify that the above estimator of $\sigma_{h_1}^2$ in~\eqref{eq:sigma_h1} can be rewritten as
\begin{equation}
\label{eq:sigma_h1n}
\sigma_{h_1,n}^2 = \frac{1}{n} \sum_{i,j=1}^n \varphi \left(\frac{i-j}{\ell_n}\right) h_1(\vec X_i) h_1(\vec X_j),
\end{equation}
and thus has the form of the HAC kernel estimator of \cite{deJDav00}.

Additionally to the conditions of Proposition~\ref{prop:multUn}, suppose that we have $b > 3 (4+\delta)/(2+\delta)$, that $\varphi$ in Definition~\ref{defn:mult_seq} is twice continuously differentiable on $(-1,1)$ with $\varphi''(0)\ne 0$ and $\sup_{x \in (-1,1)} \varphi''(x) < \infty$,  and  that $\varphi$ is Lipschitz continuous on $\R$. Then, adapting the proofs of Propositions 5.1 and 5.2 in \cite{BucKoj14} (see also Lemmas~3.12 and~3.13 in \citealp{Buh93} and Proposition 2.1 in \citealp{Sha10}), we obtain that
\begin{equation}
\label{eq:bias_var}
\Ex( \sigma_{h_1,n}^2 ) - \sigma_{h_1}^2 = \frac{\Gamma}{\ell_n^2} + o(\ell_n^{-2}) \qquad \mbox{and} \qquad \Var( \sigma_{h_1,n}^2  ) = \frac{\ell_n}{n} \Delta +  o(\ell_n/n),
\end{equation}
where $\Gamma = \varphi''(0)/2  \sum_{k=-\infty}^\infty k^2 \gamma(k)$ with $\gamma(k) = \Cov\{ h_1(\vec X_0),  h_1(\vec X_{|k|}) \}$, and where $\Delta = 2 \sigma_{h_1}^4 \int_{-1}^1 \varphi(x)^2 \dd x$. As a consequence, the mean squared error of $\sigma_{h_1,n}^2$ is
\begin{equation}
\label{eq:MSE}
\MSE ( \sigma_{h_1,n}^2 ) = \frac{\Gamma^2 }{\ell_n^4} + \Delta \frac{\ell_n}{n} + o(\ell_n^{-4}) + o(\ell_n/n).
\end{equation}
It follows that the value of $\ell_n$ that minimizes the mean square error of $\sigma_{h_1,n}^2$ is, asymptotically, 
\begin{equation}
\label{eq:lnopt}
\ell_n^{opt} = \left( \frac{4 \Gamma^2 }{\Delta} \right)^{1/5} n^{1/5}.
\end{equation}
To estimate $\ell_n^{opt}$, we first estimate the sequence $h_1(\vec X_1),\dots,h_1(\vec X_n)$ by the pseudo-observations $\hat h_{1,1:n}(\vec X_1),\dots,\hat h_{1,1:n}(\vec X_n)$, where $\hat h_{1,1:n}(\vec X_i)$ is defined as in~\eqref{eq:hath1kl}. Then, we adapt the approach of \citet{PapPol01} \citep[see also][]{PolWhi04} to the current context: let $\hat \gamma_n(k)$ be the sample autocovariance at lag $k$ computed from $\hat h_{1,1:n}(\vec X_1),\dots,\hat h_{1,1:n}(\vec X_n)$ and estimate $\Gamma$ and $\Delta$ by
$$
\hat \Gamma_n = \varphi''(0)/2  \sum_{k=-L_n}^{L_n} \lambda(k/L_n) k^2 \hat \gamma_n(k)
$$
and
$$ 
\hat \Delta_n = 2 \left\{ \sum_{k=-L_n}^{L_n} \lambda(k/L_n) \hat \gamma_n(k) \right\}^2 \left\{ \int_{-1}^1 \varphi(x)^2 \dd x \right\},
$$
respectively, where $\lambda(x) = [ \{ 2(1-|x|) \} \vee 0 ] \wedge 1$, $x \in \R$, is the ``flat top'' (trapezoidal) kernel of \cite{PolRom95} and $L_n$ is the smallest integer $k$ after which $\hat \rho_n(k)$, the sample autocorrelation at lag $k$ estimated from $\hat h_{1,1:n}(\vec X_1),\dots,\hat h_{1,1:n}(\vec X_n)$, appears negligible. The latter is determined automatically by means of the algorithm described in detail in \citet[Section 3.2]{PolWhi04}. Our implementation is based on Matlab code by A.J. Patton (available on his web page) and its \textsf{R} version by J. Racine and C. Parmeter. The resulting estimate of $\ell_n^{opt}$ in~\eqref{eq:lnopt} is denoted by $\hat \ell_n^{opt}$ as we continue.

\subsection{Applications to confidence intervals for $\theta$}
\label{sec:confint_theta}

A first straightforward application of the previous results is the computation of confidence intervals for $\theta$ in~\eqref{eq:theta}. To fix ideas, we consider three possible kernels: 
\begin{gather}
\label{eq:ker_ef}
e(x,y) = (x-y)^2/2, \qquad f(x,y) = |x-y|, \qquad \mbox{for } x,y \in \R, \\
\label{eq:ker_g}
g(\vec x,\vec y) = \1(\vec x < \vec y) + \1(\vec y < \vec x), \qquad \mbox{for } \vec x, \vec y \in \R^d.
\end{gather}
The kernels $e$ and $f$ are $\Pr$-Lipschitz-continuous as verified in \citet[Example 1.5]{DehWen10a} and \citet[Example 1.8]{DehWen10b}, respectively, while the kernel $g$ satisfies the variation condition provided the c.d.f.\ of the distribution of $\vec X_1$ is Lipschitz continuous \citep[][Appendix~C]{DehVogWenWie14}. If $d=1$ and $h=e$ (resp. $h=f$), $\theta$ is the variance of $X_1$ (resp.\ the population version of Gini's mean difference). If $d \geq 2$, the distribution of $\vec X_1$ has continuous margins and $h=g$, $\theta$ is, up to a simple affine linear transformation, a natural multivariate extension of Kendall's tau \citep{Joe90}. 

To obtain a confidence interval for $\theta$ given a sequence of suitably weakly dependent observations, a first natural possibility is to use the fact that, according to Proposition~\ref{prop:weakUn}, $\U_n(1)$ is asymptotically centered normal with variance $4\sigma_{h_1}^2$ given in~\eqref{eq:sigma_h1}. In the context under consideration, a natural estimator of $\sigma_{h_1}^2$ is~\eqref{eq:sigma_h1n}, in which $h_1(X_i)$ is estimated by $\hat h_{1,1:n}(X_i)$ as defined in~\eqref{eq:hath1kl} and in which the parameter $\ell_n$ gets replaced by the estimator $\hat \ell_n^{opt}$ introduced in the previous section. We shall denote this estimator by $\hat \sigma_{\hat h_{1,1:n}}^2$ as we continue. The resulting confidence interval of asymptotic level $1-\alpha$ is then
$$
\mathrm{CI}_{1,n} = \left[ U_{h,1:n} \pm \Phi^{-1}(1 - \alpha/2) n^{-1/2} 2 \hat \sigma_{\hat h_{1,1:n}} \right],
$$
where $\Phi$ denotes the c.d.f.\ of the standard normal distribution.

A second possibility consists of basing confidence interval on empirical quantiles computed from a sample of $M$ bootstrap replicates of $\U_n(1)$. In the studied setting, the latter involves generating $M$ independent copies of a dependent multiplier sequence 
and computing $\hat \U_n^{(1)}(1),\dots,\hat \U_n^{(M)}(1)$, where $\hat \U_n^{(m)}$ is defined in~\eqref{eq:hatUnm} (notice that $\hat \U_n^{(m)}(1) = \check \U_n^{(m)}(1)$). The resulting confidence interval of asymptotic level $1-\alpha$ is then
$$
\mathrm{CI}_{2,n} = \left[ U_{h,1:n} - n^{-1/2} \hat \U_n^{(1-\alpha/2)(M+1):M}(1), U_{h,1:n} - n^{-1/2}  \hat \U_n^{\alpha/2(M+1):M}(1) \right],
$$
where $\hat \U_n^{1:M}(1),\dots,\hat \U_n^{M:M}(1)$ are the order statistics obtained from $\hat \U_n^{(1)}(1),\dots,\hat \U_n^{(M)}(1)$. The above confidence interval is related to the so-called {\em basic bootstrap confidence interval} \citep[see, e.g.,][Chapter 5]{DavHin97}. The fact that $\mathrm{CI}_{2,n}$ is of asymptotic level $1-\alpha$ can be easily verified by combining Proposition~\ref{prop:multUn} with Proposition F.1 in \cite{BucKoj14}: Under the conditions of Proposition~\ref{prop:multUn}, as $n \to \infty$ followed by $M \to \infty$, $\Pr( \theta \in \mathrm{CI}_{2,n} )$ tends to $1-\alpha$. From a practical perspective, a natural possibility is to generate the required dependent multiplier sequences with $\ell_n = \hat \ell_n^{\scriptscriptstyle opt}$.

The computation of $\hat \ell_n^{\scriptscriptstyle opt}$, $\mathrm{CI}_{1,n}$ and $\mathrm{CI}_{2,n}$ requires the choice of the function $\varphi$ introduced in Definition~\ref{defn:mult_seq}. Following~\cite{BucKoj14}, throughout the paper, we opted for the function 
\begin{equation}
\label{eq:parzen_varphi}
x \mapsto (\kappa_P \star \kappa_P)(2x) / (\kappa_P \star \kappa_P)(0), 
\end{equation}
where `$\star$' denotes the convolution operator and $\kappa_P$ is the Parzen kernel, that is,
\begin{equation}
\label{eq:parzen}
\kappa_{P}(x) = (1 - 6x^2 + 6|x|^3) \1(|x| \leq 1/2) + 2(1-|x|)^3\1(1/2 < |x| \leq 1), \qquad x \in \R.
\end{equation}
The latter choice is theoretically sensible in view of~\eqref{eq:bias_var} and~\eqref{eq:MSE}, and was also found to lead to good finite-sample performance in the numerical experiments presented in \citet[Section 6]{BucKoj14}.

\setlength{\tabcolsep}{4pt}
\begin{table}[t!]
\centering
\caption{For $h \in \{e,f\}$, coverage percentages of $\mathrm{CI}_{1,n}$ and $\mathrm{CI}_{2,n}$  estimated from 2,000 univariate samples of size $n$ generated from an AR1 model with parameter $\zeta \in \{0,0.5,0.9\}$ and either standard normal (first horizontal block) or $t_5$ (second horizontal block) innovations.} 
\label{covprob}
{\small
\begin{tabular}{rrrrrrrrrrrrrr}
  \hline
  \multicolumn{2}{c}{} & \multicolumn{6}{c}{variance ($h=e$)} & \multicolumn{6}{c}{Gini's mean diff. ($h=f$)}  \\
\cmidrule(lr){3-8} \cmidrule(lr){9-14}
\multicolumn{2}{c}{}& \multicolumn{2}{c}{$\zeta=0$}  & \multicolumn{2}{c}{$\zeta=0.5$} & \multicolumn{2}{c}{$\zeta=0.9$} & \multicolumn{2}{c}{$\zeta=0$} & \multicolumn{2}{c}{$\zeta=0.5$} & \multicolumn{2}{c}{$\zeta=0.9$} \\
\cmidrule(lr){3-4} \cmidrule(lr){5-6} \cmidrule(lr){7-8} \cmidrule(lr){9-10} \cmidrule(lr){11-12} \cmidrule(lr){13-14} $\alpha$ & $n$ & $\mathrm{CI}_{1,n}$ & $\mathrm{CI}_{2,n}$ & $\mathrm{CI}_{1,n}$ & $\mathrm{CI}_{2,n}$ & $\mathrm{CI}_{1,n}$ & $\mathrm{CI}_{2,n}$ & $\mathrm{CI}_{1,n}$ & $\mathrm{CI}_{2,n}$ & $\mathrm{CI}_{1,n}$ & $\mathrm{CI}_{2,n}$ & $\mathrm{CI}_{1,n}$ & $\mathrm{CI}_{2,n}$ \\ \hline
0.10 & 25 & 78.0 & 78.6 & 67.0 & 68.0 & 20.8 & 21.7 & 81.6 & 81.5 & 73.0 & 74.1 & 26.6 & 27.4 \\ 
   & 50 & 83.0 & 83.4 & 74.8 & 75.8 & 39.3 & 40.4 & 86.4 & 86.7 & 79.7 & 80.4 & 43.9 & 44.9 \\ 
   & 100 & 84.7 & 85.0 & 81.3 & 82.0 & 56.0 & 56.3 & 88.6 & 88.7 & 84.0 & 84.5 & 62.2 & 63.1 \\ 
   & 200 & 87.7 & 87.6 & 85.8 & 86.2 & 69.6 & 69.9 & 88.2 & 88.4 & 85.2 & 85.9 & 68.9 & 69.1 \\ 
  0.05 & 25 & 85.5 & 86.0 & 76.4 & 77.3 & 24.9 & 25.6 & 88.2 & 88.0 & 79.4 & 79.7 & 31.3 & 32.3 \\ 
   & 50 & 89.2 & 89.4 & 81.5 & 82.1 & 45.3 & 46.1 & 91.2 & 91.3 & 85.9 & 86.9 & 54.0 & 55.0 \\ 
   & 100 & 92.4 & 92.3 & 89.4 & 89.9 & 62.7 & 63.3 & 92.9 & 93.1 & 88.4 & 88.9 & 70.2 & 70.8 \\ 
   & 200 & 93.9 & 94.1 & 91.1 & 91.4 & 73.6 & 73.8 & 93.6 & 93.8 & 91.3 & 91.7 & 76.7 & 76.9 \\ 
  0.01 & 25 & 91.4 & 91.7 & 81.4 & 82.2 & 33.4 & 34.7 & 94.6 & 94.2 & 89.3 & 89.5 & 42.6 & 44.2 \\ 
   & 50 & 95.5 & 95.6 & 87.7 & 88.3 & 47.7 & 48.5 & 97.0 & 97.0 & 91.9 & 92.5 & 61.1 & 62.2 \\ 
   & 100 & 96.8 & 96.9 & 93.9 & 94.1 & 67.8 & 68.3 & 97.7 & 97.6 & 95.2 & 95.4 & 76.5 & 77.3 \\ 
   & 200 & 97.9 & 98.0 & 95.2 & 95.4 & 82.9 & 83.5 & 98.8 & 98.7 & 97.2 & 97.4 & 87.8 & 87.9 \\ 
   \hline0.10 & 25 & 70.2 & 71.1 & 62.1 & 63.4 & 19.2 & 20.0 & 80.7 & 81.2 & 69.8 & 70.6 & 27.1 & 27.9 \\ 
   & 50 & 79.1 & 79.6 & 69.3 & 70.6 & 39.1 & 40.2 & 85.1 & 85.6 & 76.6 & 77.6 & 44.6 & 45.6 \\ 
   & 100 & 81.7 & 82.0 & 76.3 & 77.3 & 54.0 & 54.6 & 85.2 & 85.5 & 80.6 & 81.3 & 56.6 & 57.2 \\ 
   & 200 & 85.7 & 85.8 & 81.9 & 82.3 & 61.1 & 61.6 & 89.2 & 89.3 & 85.3 & 85.8 & 69.6 & 70.2 \\ 
  0.05 & 25 & 75.7 & 76.7 & 67.0 & 68.9 & 27.3 & 28.2 & 87.2 & 87.5 & 77.0 & 78.3 & 33.6 & 34.8 \\ 
   & 50 & 83.3 & 83.5 & 73.6 & 74.4 & 43.0 & 44.3 & 90.7 & 91.0 & 83.9 & 84.6 & 49.8 & 50.8 \\ 
   & 100 & 86.0 & 86.3 & 82.3 & 82.8 & 61.8 & 62.5 & 91.8 & 91.9 & 86.0 & 86.6 & 65.8 & 66.5 \\ 
   & 200 & 88.3 & 88.3 & 89.4 & 89.6 & 69.6 & 70.1 & 93.3 & 93.3 & 91.4 & 91.9 & 76.9 & 77.3 \\ 
  0.01 & 25 & 85.5 & 86.0 & 76.5 & 77.4 & 33.2 & 34.5 & 92.3 & 92.3 & 84.4 & 84.9 & 40.8 & 42.1 \\ 
   & 50 & 90.8 & 91.2 & 83.2 & 84.1 & 50.2 & 51.2 & 95.2 & 95.3 & 91.3 & 91.8 & 60.8 & 62.0 \\ 
   & 100 & 91.1 & 91.3 & 88.2 & 88.6 & 64.7 & 65.3 & 97.4 & 97.5 & 94.9 & 95.2 & 77.0 & 77.5 \\ 
   & 200 & 94.9 & 94.9 & 94.4 & 94.5 & 78.7 & 79.0 & 97.4 & 97.4 & 95.8 & 96.0 & 86.0 & 86.2 \\ 
   \hline
\end{tabular}
}
\end{table}

As a brief illustration, Table~\ref{covprob} reports coverage percentages when $h \in \{e,f\}$ of $\mathrm{CI}_{1,n}$ and $\mathrm{CI}_{2,n}$ estimated from 2,000 univariate samples of size $n$ generated from an AR1 model with parameter $\zeta \in \{0,0.5,0.9\}$ and either standard normal or $t_5$ innovations.    The setting $\zeta=0$ (resp.\ $\zeta=0.5$, $\zeta = 0.9$) corresponds to serial independence (resp.\ moderate, strong) serial dependence. Although this was not always necessary, the true value of $\theta$ was estimated from a sample of size 20,000 using~\eqref{eq:Uh1n}. The number of multiplier bootstrap replicates  necessary to compute $\mathrm{CI}_{2,n}$ was set to $M=4999$. The corresponding dependent multiplier sequences were generated using the ``moving average approach'' proposed initially in \citet[Section~6.2]{Buh93} and revisited in \citet[Section~5.2]{BucKoj14}. A standard normal sequence was used for the required initial i.i.d.\ sequence. The kernel function $\kappa$ in that procedure was chosen to be the Parzen kernel $\kappa_P$ defined in~\eqref{eq:parzen}, which amounts to choosing~\eqref{eq:parzen_varphi} for the function $\varphi$ in Definition~\ref{defn:mult_seq}.

As one can see, $\mathrm{CI}_{1,n}$ and $\mathrm{CI}_{2,n}$ are too narrow for the sample sizes under consideration. Unsurprisingly, the coverage rates are particularly poor for small $n$ and strong serial dependence ($\zeta = 0.9$). In all settings under serial dependence ($\zeta \in \{0.5,0.9\}$), $\mathrm{CI}_{2,n}$ displays better coverage rates than $\mathrm{CI}_{1,n}$. The difference, as expected, decreases as $n$ increases. As observed in other settings, the use of {\em studentized} bootstrap confidence intervals \citep[see, e.g.,][Chapter 5]{DavHin97} could lead to improved coverage rates. The latter would require the availability of a resampling scheme for the estimator $\hat \sigma_{\scriptscriptstyle \hat h_{1,1:n}}^2$ of $\sigma_{h_1}^2$ involved in the expression of~$\mathrm{CI}_{1,n}$ and, as mentioned by a reviewer, may not be without ambiguity under serial dependence. Finally, note that for $t_5$ innovations and $h=e$, the moments conditions on the kernel in Proposition~\ref{prop:multUn} are not satisfied. Still, the finite-sample behavior of~$\mathrm{CI}_{2,n}$ relatively to that of~$\mathrm{CI}_{1,n}$ does not seem affected suggesting that Proposition~\ref{prop:multUn} might hold under weaker conditions.

Finally, the presented application also highlights the fact that multiplier bootstrap procedures can often be implemented to be computationally efficient. In the setting under consideration, the sample $\hat h_{1,1:n}(\vec X_1),\dots,\hat h_{1,1:n}(\vec X_n)$ required for computing~\eqref{eq:hatUnm} needs to be computed only once. The computational cost for obtaining the multiplier replicates $\hat \U_n^{(1)}(1),\dots,\hat \U_n^{(M)}(1)$ then essentially boils down to that of the generation of the required $M$ dependent multiplier sequences. The latter seems very reasonable when based on the ``moving average approach'' of \citet{Buh93} discussed above.

\section{Two dependent multiplier bootstraps for $\D_n$} 
\label{sec:multDn}

\subsection{Dependent multiplier results for $\D_n$}

Results analogous to Propositions~\ref{prop:weakUn} and~\ref{prop:multUn} can be obtained for the process $\D_n$ in~\eqref{eq:Dn}. As we shall see, they have immediate applications to change-point detection. The starting point for deriving such results is to note that, for any $s \in [2/n,1-2/n]$, $\D_n$ can be rewritten as
\begin{equation}
\label{eq:DnH0}
\D_n(s) = \lambda_n(s,1) \U_n(s) -  \lambda_n(0,s) \U_n^*(s),
\end{equation}
where $\U_n$ is defined in~\eqref{eq:Un} and where, for any $s \in [0,1]$,
\begin{equation}
\label{eq:Un*}
\U_n^*(s) = \left\{ 
\begin{array}{ll}
\sqrt{n} \lambda_n(s,1) \{ U_{h, \ip{ns}+1:n} - \theta \},
&\mbox{if } s \in [0,1-2/n], \\
0, &\mbox{otherwise}.
\end{array} 
\right.
\end{equation}
The following proposition extends Theorem~1.1 of \cite{GomHor02} to mixing sequences. Its proof is given in Appendix~\ref{proofs:Dn}. 

\begin{prop}[Asymptotics of $\D_n$]
\label{prop:weakDn}
Under the conditions of Proposition~\ref{prop:weakUn}, $\D_n \leadsto \D$ in $\ell^\infty([0,1])$, where $\D(s) =2\sigma_{h_1} \{ \B(s) - s \B(1) \}$, $s \in [0,1]$, with $\B$ a standard Brownian motion and $\sigma_{h_1}$ given in~\eqref{eq:sigma_h1}.
\end{prop}

To obtain dependent multiplier bootstrap results for $\D_n$ in the spirit of those obtained in the previous section, we start again from~\eqref{eq:DnH0}. Since $\U_n$ can be resampled using the processes $\hat \U_n^{(m)}$ or $\check \U_n^{(m)}$, $m \in \{1,\dots,M\}$, defined in~\eqref{eq:hatUnm} and~\eqref{eq:checkUnm}, respectively, it suffices to define the corresponding bootstrap replicates for the process $\U_n^*$ in~\eqref{eq:Un*} to obtain bootstrap replicates of $\D_n$. Thus, for any $m \in \{1,\dots,M\}$, let
$$
\hat \U_n^{*,(m)} (s) = \frac{2}{\sqrt{n}} \sum_{i=\ip{ns}+1}^n \xi_{i,n}^{(m)} \hat h_{1,1:n}(\vec X_i), \qquad s \in [0,1],
$$
and 
\begin{equation}
\label{eq:checkUn*m}
\check \U_n^{*,(m)} (s) = \frac{2}{\sqrt{n}} \sum_{i=\ip{ns}+1}^n \xi_{i,n}^{(m)} \hat h_{1,\ip{ns}+1:n}(\vec X_i), \qquad s \in [0,1],
\end{equation}
where, for any integers $1 \leq k \leq l \leq n$, $\hat h_{1,k:l}$ is defined in~\eqref{eq:hath1kl}. Corresponding dependent multiplier bootstrap replicates of $\D_n$ are then naturally given by
\begin{equation}
\label{eq:hatDnm}
\hat \D_n^{(m)} (s) = \lambda_n(s,1) \hat \U_n^{(m)} (s) - \lambda_n(0,s) \hat \U_n^{*,(m)} (s),
\end{equation}
and
\begin{equation}
\label{eq:checkDnm}
\check \D_n^{(m)} (s) = \lambda_n(s,1) \check \U_n^{(m)} (s) - \lambda_n(0,s) \check \U_n^{*,(m)} (s).
\end{equation}

A proof of the following result is given in Appendix~\ref{proofs:Dn}.

\begin{prop}[Two dependent multiplier bootstraps for $\D_n$]
\label{prop:multDn}
Under the conditions of Proposition~\ref{prop:multUn},
$$
\left(\D_n,\hat \D_n^{(1)},\dots,\hat \D_n^{(M)},\check \D_n^{(1)},\dots,\check \D_n^{(M)} \right) \leadsto \left(\D,\D^{(1)},\dots,\D^{(M)},\D^{(1)},\dots,\D^{(M)} \right)
$$
in $\{\ell^\infty([0,1])\}^{2M+1}$,  where $\D$ is the weak limit of $\D_n$ given in Proposition~\ref{prop:weakDn}, and $\D^{(1)},\dots,\D^{(M)}$ are independent copies of $\D$.
\end{prop} 

\subsection{Applications to change-point detection} \label{subsec:cp}

The previous result is of immediate interest in the context of tests for change-point detection. Recall that the aim of such statistical procedures is to test
\begin{equation}
\label{H0}
  H_0 : \,\exists \, F \text{ such that } 
  \vec X_1, \ldots, \vec X_n \text{ have c.d.f. } F
\end{equation}
against alternatives involving the non-constancy of the c.d.f.\ \citep[see, e.g.,][for an overview of possible approaches]{CsoHor97}. As already mentioned, a typical test statistic is $S_n$ in~\eqref{eq:Sn}. To fix ideas, we consider again the kernels defined in~\eqref{eq:ker_ef} and~\eqref{eq:ker_g}. Choosing $h=e$ (resp. $h=f$) results in tests for $H_0$ that are particularly sensitive to changes in the variance (resp.\ Gini's mean difference) of the observations. The choice $h=g$ leads to tests for $H_0$ particularly sensitive to changes in the cross-correlation of multivariate time series as measured by Kendall's tau. Such tests were studied by \cite{QueSaiFav13} in the case of serially independent observations and, more recently, by \cite{DehVogWenWie14} in the case of {\em $\Pr$-near epoch dependent} sequences.

The usual way of carrying out tests based on $S_n$ is to exploit the fact that, under $H_0$ and for instance the conditions of Proposition~\ref{prop:weakUn}, $S_n$ converges weakly to $S = 2 \sigma_{h_1} \sup_{s \in [0,1]} | \B(s) - s \B(1)|$, where $\B$ is a standard Brownian motion. In other words, $S_n \sigma_{h_1}^{-1} / 2$ converges weakly to the supremum of a Brownian bridge, which implies that its limiting distribution is the Kolmogorov distribution. The c.d.f.\ $F_K$ of the latter distribution can be approximated very well numerically. In the setting considered in this work, it is thus natural to compute an approximate p-value for $S_n$ as 
\begin{equation}
\label{eq:asym_pval}
1 - F_K(S_n \hat \sigma_{\hat h_{1,1:n}}^{-1} / 2),
\end{equation} 
where, again, $\hat \sigma_{\scriptscriptstyle \hat h_{1,1:n}}^{\scriptscriptstyle 2}$ is the estimator of $\sigma_{h_1}^2$ obtained from~\eqref{eq:sigma_h1n} by replacing $h_1(X_i)$ by $\hat h_{1,1:n}(X_i)$ as defined in~\eqref{eq:hath1kl}, and in which the bandwidth parameter $\ell_n$ is estimated by $\hat \ell_n^{opt}$ (see Section~\ref{sec:estim-bandw-param}). As mentioned by a reviewer, the latter choice for $\ell_n$ is optimal only in the context of the estimation of the long-run variance, and a test-optimal bandwidth \citep[see][]{SunPhiJin08} could be investigated in future work.

An alternative way to carry out the test consists of resampling $S_n$. For any $m \in \{1,\dots,M\}$, let
$$
\hat S_n^{(m)} = \sup_{s \in [0,1]} |\hat \D_n^{(m)} (s) | \qquad \mbox{and} \qquad \check S_n^{(m)} = \sup_{s \in [0,1]} |\check \D_n^{(m)} (s) |,
$$
where $\hat \D_n^{(m)}$ and $\check \D_n^{(m)}$ are defined in~\eqref{eq:hatDnm} and~\eqref{eq:checkDnm}, respectively. From Proposition~\ref{prop:multDn} and the continuous mapping theorem, we then immediately have that, under $H_0$ and the conditions of Proposition~\ref{prop:multUn},
\begin{equation}
\label{eq:multSn}
\left(S_n,\hat S_n^{(1)},\dots,\hat S_n^{(M)},\check S_n^{(1)},\dots,\check S_n^{(M)} \right) \leadsto \left(S,S^{(1)},\dots,S^{(M)},S^{(1)},\dots,S^{(M)} \right)
\end{equation}
in $\R^{2M+1}$, where $S$ is the weak limit of $S_n$, and $S^{(1)},\dots,S^{(M)}$ are independent copies of~$S$. The previous result suggests computing an approximate p-value for $S_n$ as
\begin{equation}
\label{eq:pval}
\frac{1}{M} \sum_{m=1}^M \1 \left( \hat S_n^{(m)} \geq S_n \right) \quad \mbox{or as} \quad 
\frac{1}{M} \sum_{m=1}^M \1 \left( \check S_n^{(m)} \geq S_n \right).
\end{equation}
The weak convergence in~\eqref{eq:multSn} can be combined with Proposition F.1 in \cite{BucKoj14} to show that a test based on $S_n$ whose p-value is computed as in~\eqref{eq:pval} will hold its level asymptotically as $n \to \infty$ followed by $M \to \infty$.

\subsection{Monte Carlo experiments}
\label{sec:monte-carlo-exper}

To illustrate the previous developments, we restrict our attention to the choice $h=g$ and $d=2$, that is, to tests for change-point detection for bivariate data that are particularly sensitive to changes in Kendall's tau. The aim is to compare the two ways for carrying out the test discussed previously for samples of moderate size. When the approximate p-value for $S_n$ in~\eqref{eq:Sn} is computed using~\eqref{eq:asym_pval}, we shall talk about {\em the test based on $S_n^{\hat \sigma}$}, while when it is based on~\eqref{eq:pval}, we shall talk about {\em the test based on $\hat S_n$} or {\em the test based on $\check S_n$}.

Two simple time series models were used to generate bivariate samples of size $n$. Given a real $t \in (0,1)$ determining the location of the possible change-point of the innovations, two bivariate copulas $C_1$ and $C_2$, and parameters $\zeta$, $\vec \omega, \vec \beta, \vec \alpha$ to be specified below, the following steps were followed to generate a bivariate sample $\vec X_1,\dots,\vec X_n$:
\begin{enumerate}
\item generate independent bivariate random vectors $\vec U_i$, $i \in \{-100,\dots,0,\dots,n\}$ such that $\vec U_i$, $i \in \{-100,\dots,0,\dots,\ip{nt}\}$ are i.i.d.\ from copula $C_1$ and $\vec U_i$, $i \in \{\ip{nt}+1,\dots,n\}$ are i.i.d.\ from copula $C_2$,
\item compute $\vec \epsilon_i = (\Phi^{-1}(U_{i1}),\Phi^{-1}(U_{i2}))$, where $\Phi$ is the c.d.f.\ of the standard normal distribution, 
\item set $\vec X_{-100} = \vec \epsilon_{-100}$ and, for $j = 1,2$, compute recursively either
\begin{equation}
\tag{AR1}
\label{eq:ar1}
  X_{ij} = \zeta X_{i-1,j} + \epsilon_{ij}, \qquad i=-99,\dots,0,\dots,n,
\end{equation} 
or
\begin{equation}
\tag{GARCH}
\label{eq:garch}
\sigma_{ij}^2 = \omega_j + \beta_j \sigma_{i-1,j}^2 + \alpha_j \epsilon_{i-1,j}^2, \qquad X_{ij} = \sigma_{ij} \epsilon_{ij}, \qquad i=-99,\dots,0,\dots,n.
\end{equation}
\end{enumerate}
If the copulas $C_1$ and $C_2$ are chosen equal, the above procedure generates samples under~$H_0$ defined in~\eqref{H0}. Three possible values were considered for the parameter $\zeta$ controlling the strength of the serial dependence in~\eqref{eq:ar1}: 0 (serial independence), 0.25 (weak serial dependence), 0.5 (moderate serial dependence). For model~\eqref{eq:garch}, following \cite{BucRup13}, we took $(\omega_1,\beta_1,\alpha_1)=(0.012,0.919,0.072)$ and $(\omega_2,\beta_2,\alpha_2)=(0.037,0.868,0.115)$. The latter values were estimated by \cite{JonPooRoc07} from SP500 and DAX daily logreturns, respectively. 

Two one-parameter copula families were considered: the Clayton (which is {\em upper-tail dependent}) and the Gumbel--Hougaard (which is {\em lower-tail dependent}) \cite[see, e.g.,][]{Nel06}. For both families, there exists a one-to-one relationship between the parameter value and Kendall's tau. To estimate the power of the tests, 1,000 samples were generated under each combination of factors and all the tests were carried out at the 5\% significance level. For the tests based on $\hat S_n$ or $\check S_n$, $M=1,000$ multiplier replications were used. The corresponding dependent multiplier sequences were generated as explained in Section~\ref{sec:confint_theta} with $\ell_n = \hat \ell_n^{opt}$. For the test based on $S_n^{\hat \sigma}$, as classically done, we approached the c.d.f.\ $F_k$ in~\eqref{eq:asym_pval} by that of the statistic of the classical Kolmogorov--Smirnov goodness-of-fit test for a simple hypothesis. From a practical perspective, we used the function {\tt pkolmogorov1x} given in the code of the \textsf{R} function {\tt ks.test}.

\begin{table}[t!]
\centering
\caption{Percentage of rejection of $H_0$ computed from 1,000 samples of size $n \in \{50, 100, 200\}$ when $C_1 = C_2 = C$ is either the bivariate Clayton (Cl) or the Gumbel--Hougaard (GH) copula with a Kendall's tau of $\tau$.} 
\label{resH0}
{\small
\begin{tabular}{lrrrrrrrrrr}
  \hline
  \multicolumn{3}{c}{} & \multicolumn{2}{c}{$\zeta=0$} & \multicolumn{2}{c}{$\zeta=0.25$} & \multicolumn{2}{c}{$\zeta=0.5$} & \multicolumn{2}{c}{GARCH} \\ \cmidrule(lr){4-5} \cmidrule(lr){6-7} \cmidrule(lr){8-9} \cmidrule(lr){10-11} $C$ & $n$ & $\tau$ & $\check S_n$ & $S_n^{\hat \sigma}$ & $\check S_n$ & $S_n^{\hat \sigma}$ & $\check S_n$ & $S_n^{\hat \sigma}$ & $\check S_n$ & $S_n^{\hat \sigma}$ \\ \hline
Cl & 50 & 0.1 & 6.0 & 5.6 & 5.3 & 6.1 & 11.7 & 8.4 & 6.2 & 6.0 \\ 
   &  & 0.3 & 5.3 & 5.3 & 6.2 & 7.2 & 8.5 & 6.8 & 5.8 & 6.0 \\ 
   &  & 0.5 & 4.3 & 7.9 & 5.7 & 10.7 & 6.7 & 11.5 & 4.4 & 8.7 \\ 
   &  & 0.7 & 4.7 & 16.4 & 4.1 & 17.4 & 8.4 & 16.3 & 5.6 & 16.2 \\ 
   & 100 & 0.1 & 6.2 & 5.5 & 5.3 & 4.4 & 7.8 & 5.1 & 6.0 & 4.6 \\ 
   &  & 0.3 & 5.6 & 5.6 & 6.1 & 5.8 & 8.1 & 6.6 & 4.2 & 4.2 \\ 
   &  & 0.5 & 4.3 & 4.7 & 5.3 & 6.1 & 6.3 & 6.7 & 4.9 & 5.9 \\ 
   &  & 0.7 & 2.5 & 10.4 & 2.2 & 8.4 & 4.9 & 9.7 & 2.7 & 9.4 \\ 
   & 200 & 0.1 & 6.4 & 5.0 & 6.5 & 5.3 & 5.5 & 3.8 & 5.8 & 5.1 \\ 
   &  & 0.3 & 4.9 & 4.4 & 5.0 & 4.4 & 7.2 & 5.3 & 7.5 & 6.2 \\ 
   &  & 0.5 & 4.8 & 4.7 & 5.0 & 5.5 & 6.3 & 5.8 & 6.2 & 6.0 \\ 
   &  & 0.7 & 3.7 & 5.9 & 3.8 & 5.8 & 5.5 & 7.0 & 4.2 & 5.9 \\ 
  GH & 50 & 0.1 & 5.4 & 4.9 & 7.3 & 6.1 & 9.7 & 7.0 & 6.5 & 5.1 \\ 
   &  & 0.3 & 5.8 & 5.8 & 6.3 & 6.7 & 7.7 & 8.1 & 6.0 & 6.7 \\ 
   &  & 0.5 & 4.8 & 8.8 & 5.2 & 8.5 & 8.5 & 11.3 & 6.4 & 10.6 \\ 
   &  & 0.7 & 6.3 & 20.6 & 5.7 & 20.2 & 6.8 & 18.9 & 5.0 & 18.7 \\ 
   & 100 & 0.1 & 5.0 & 4.1 & 6.5 & 5.3 & 8.9 & 6.1 & 5.6 & 4.8 \\ 
   &  & 0.3 & 4.8 & 4.4 & 6.1 & 5.8 & 7.7 & 6.9 & 5.5 & 4.8 \\ 
   &  & 0.5 & 4.4 & 4.9 & 4.2 & 5.7 & 7.1 & 8.8 & 4.1 & 5.3 \\ 
   &  & 0.7 & 3.3 & 10.6 & 3.0 & 9.3 & 4.3 & 11.0 & 3.8 & 9.0 \\ 
   & 200 & 0.1 & 5.7 & 5.1 & 5.7 & 4.4 & 7.0 & 5.2 & 6.1 & 4.5 \\ 
   &  & 0.3 & 5.8 & 5.4 & 4.8 & 4.4 & 6.6 & 5.9 & 6.7 & 5.8 \\ 
   &  & 0.5 & 3.7 & 4.9 & 4.8 & 4.6 & 7.5 & 6.6 & 5.0 & 5.4 \\ 
   &  & 0.7 & 2.8 & 6.1 & 3.9 & 6.6 & 4.9 & 7.8 & 3.3 & 5.4 \\ 
   \hline
\end{tabular}
}
\end{table}

Table~\ref{resH0} reports the rejection percentages of $H_0$ for observations generated under the null. To ease reading, the rejection rates of the test based on $\hat S_n$ are not reported as the latter turned out, overall, to be worse behaved than the test based on $\check S_n$ for the sample sizes under consideration. As one can see, the test based on  $S_n^{\hat \sigma}$ tends to be way too liberal when the cross-sectional dependence is high ($\tau \geq 0.5$), although its behavior improves as $n$ increases. The empirical levels of the test based on $\check S_n$ are overall reasonably good when $\zeta \in \{0,0.25\}$ and for sequences generated using~\eqref{eq:garch}, even for small sample sizes. However, under stronger serial dependence corresponding to $\zeta=0.5$ in~\eqref{eq:ar1}, the test is overall too liberal although the agreement with the 5\% nominal level improves as $n$ increases.

\begin{table}[t!]
\centering
\caption{Percentage of rejection of $H_0$ computed from 1,000 samples of size $n \in \{50, 100, 200\}$ when $C_1$ and $C_2$ are both bivariate Gumbel--Hougaard copulas such that $C_1$ has a Kendall's tau of 0.2 and $C_2$ has a Kendall's tau of $\tau$.} 
\label{resH1}
{\small
\begin{tabular}{rrrrrrrrrrr}
  \hline
  \multicolumn{3}{c}{} & \multicolumn{2}{c}{$\zeta=0$} & \multicolumn{2}{c}{$\zeta=0.25$} & \multicolumn{2}{c}{$\zeta=0.5$} & \multicolumn{2}{c}{GARCH}  \\ \cmidrule(lr){4-5} \cmidrule(lr){6-7} \cmidrule(lr){8-9} \cmidrule(lr){10-11} $n$ & $t$ & $\tau$ & $\check S_n$ & $S_n^{\hat \sigma}$ & $\check S_n$ & $S_n^{\hat \sigma}$ & $\check S_n$ & $S_n^{\hat \sigma}$ & $\check S_n$ & $S_n^{\hat \sigma}$ \\ \hline
50 & 0.10 & 0.4 & 5.2 & 7.3 & 7.1 & 9.3 & 9.4 & 11.9 & 6.7 & 8.6 \\ 
   &  & 0.6 & 9.7 & 17.2 & 10.5 & 19.9 & 12.2 & 21.7 & 9.3 & 16.5 \\ 
   & 0.25 & 0.4 & 11.7 & 11.1 & 11.5 & 10.8 & 12.4 & 13.3 & 9.6 & 9.6 \\ 
   &  & 0.6 & 30.7 & 29.2 & 29.6 & 25.7 & 29.2 & 24.2 & 29.7 & 27.0 \\ 
   & 0.50 & 0.4 & 14.7 & 13.9 & 15.6 & 14.4 & 15.9 & 12.4 & 12.1 & 11.3 \\ 
   &  & 0.6 & 48.2 & 37.0 & 47.6 & 36.4 & 41.9 & 27.2 & 43.4 & 35.3 \\ 
  100 & 0.10 & 0.4 & 7.4 & 8.3 & 6.0 & 6.2 & 9.2 & 8.9 & 5.8 & 6.1 \\ 
   &  & 0.6 & 13.3 & 18.2 & 14.3 & 18.8 & 14.8 & 18.1 & 14.3 & 18.2 \\ 
   & 0.25 & 0.4 & 15.9 & 14.1 & 16.5 & 16.0 & 19.2 & 14.5 & 14.7 & 13.8 \\ 
   &  & 0.6 & 60.4 & 52.3 & 60.3 & 51.6 & 49.4 & 37.6 & 59.3 & 51.6 \\ 
   & 0.50 & 0.4 & 26.8 & 21.8 & 23.8 & 20.5 & 24.8 & 17.8 & 26.5 & 22.5 \\ 
   &  & 0.6 & 85.4 & 77.9 & 79.6 & 69.8 & 72.2 & 55.4 & 83.3 & 72.7 \\ 
  200 & 0.10 & 0.4 & 8.2 & 8.3 & 8.3 & 8.4 & 11.1 & 8.7 & 9.0 & 7.9 \\ 
   &  & 0.6 & 35.3 & 36.8 & 33.7 & 36.4 & 21.7 & 23.1 & 32.4 & 34.0 \\ 
   & 0.25 & 0.4 & 33.9 & 31.9 & 31.3 & 28.9 & 27.3 & 21.1 & 27.0 & 26.0 \\ 
   &  & 0.6 & 93.5 & 83.9 & 91.0 & 79.9 & 81.0 & 62.6 & 91.4 & 82.1 \\ 
   & 0.50 & 0.4 & 48.2 & 44.1 & 49.1 & 45.7 & 41.2 & 31.9 & 47.3 & 43.8 \\ 
   &  & 0.6 & 99.1 & 97.1 & 98.3 & 94.3 & 95.1 & 86.4 & 99.3 & 95.8 \\ 
   \hline
\end{tabular}
}
\end{table}

Table~\ref{resH1} reports the rejection percentages of $H_0$ for bivariate sequences generated with a break in the innovations whose position is determined by the parameter $t$. The results are those obtained when $C_1$ and $C_2$ are Gumbel--Hougaard copulas with different Kendall's taus. Similar results (not reported) were obtained when $C_1$ and $C_2$ are Clayton copulas instead. As one can see, the test based on $\check S_n$ seems overall more powerful except when the change in the innovations occurs early ($t=0.1$). Of course, when analyzing these results, one should keep in mind that the test based on $S_n^{\hat \sigma}$ was observed to be too liberal in the case of strong cross-sectional dependence, and that both tests displayed inflated empirical levels, overall, for $\zeta = 0.5$. 

\section{Conclusion} \label{sec:con}

Starting from the work of \cite{GomHor02} and \cite{DehWen10b}, we have studied the asymptotic behavior of sequential resampling schemes for the processes $\U_n$ and $\D_n$ defined in~\eqref{eq:Un} and~\eqref{eq:Dn}, respectively. Monte Carlo experiments indicate that the use of the derived {\em dependent multiplier bootstraps} can have advantages over that of estimated asymptotic distributions in the context of confidence interval construction or tests for change-point detection. Future work could consist of studying resampling schemes for estimators of the variance $\sigma_{h_1}^2$ in~\eqref{eq:sigma_h1} (with studentized confidence intervals in mind), or comparing dependent multiplier tests for change-point detection based on $S_n$ in~\eqref{eq:Sn} with their {\em self-normalization} version proposed in \cite{ShaZha10}.

\appendix

\section{Proof of Proposition~\ref{prop:weakUn}}
\label{proof:prop:weakUn}

Throughout this and the following proofs, we will frequently apply results from \cite{DehWen10a,DehWen10b}. The latter are stated for $d=1$ only, but actually hold true for $d>1$ as explained in \citet[Appendix B]{DehVogWenWie14}.

\begin{proof}[\bf Proof of Proof of Proposition~\ref{prop:weakUn}.]
Since $b > (2+\delta)/\delta$, Theorem~2 of \cite{OodYos72} implies that $\sigma_{h_1}^2 < \infty$ and that the process $s \mapsto n^{-1/2}  \sum_{i=1}^{\ip{ns}}  h_1(\vec X_i)$ converges weakly to $\sigma_{h_1} \B$ in $\ell^\infty([0,1])$. To show the desired result, it therefore suffices to show that $\sup_{s \in [0,1]} | \U_n(s) - 2 n^{-1/2}  \sum_{i=1}^{\ip{ns}}  h_1(\vec X_i) | = o_\Pr(1)$. Since $\U_n(s) = 0$ if $s \in [0,2/n)$ and since $2 n^{-1/2} h_1(\vec X_1) = o_\Pr(1)$, we immediately obtain that $\sup_{s \in [0,2/n)} | \U_n(s) - 2 n^{-1/2}  \sum_{i=1}^{\ip{ns}}  h_1(\vec X_i) | = o_\Pr(1)$.

For any $s \in [2/n,1]$, using Hoeffding's decomposition~\eqref{eq:Hoeff_decomp}, we obtain that
$$
\U_n(s) = \frac{2}{\sqrt{n}}  \sum_{i=1}^{\ip{ns}}  h_1(\vec X_i) + \sqrt{n} \lambda_n(0,s) U_{h_2,1:\ip{ns}}.
$$
Hence, 
$$
\sup_{s \in [2/n,1]} \left| \U_n(s) - \frac{2}{\sqrt{n}}  \sum_{i=1}^{\ip{ns}}  h_1(\vec X_i) \right| = n^{-1/2} \max_{2 \leq k \leq n} k |U_{h_2,1:k}|.
$$
It remains to show that the latter supremum converges in probability to zero. Under~(i), let $a = \delta b / (2 + \delta)$ and notice that $a > 1$ from the assumption on the mixing rate. Furthermore, if $-1 < 1-a < 0$, from well-known results on Riemann series, $\sum_{r=1}^n r \beta_r^{\delta/(2+\delta)} \le \mbox{const} \times \sum_{r=1}^n r^{1-a} = O(n^{2-a})$. Set $\tau = \max(2-a,0)$. Then, $\sum_{i=1}^n r \beta_r^{\delta/(2+\delta)} = O(n^\tau)$ and the conditions of Theorem~1 in \cite{DehWen10b} are satisfied. Let additionally $a_k = (\log k)^{3/2} \log \log k$, $k \geq 2$. Then, 
\begin{align*}
n^{-1/2} \max_{2 \leq k \leq n} k |U_{h_2,1:k}| &\leq n^{-1/2} \max_{2 \leq k \leq n} \frac{k^{1-\tau/2} |U_{h_2,1:k}|}{a_k} \times  \max_{2 \leq k \leq n} k^{\tau/2} a_k \\
&\leq \sup_{k \geq 2} \frac{k^{1-\tau/2} |U_{h_2,1:k}|}{a_k} \times  n^{\tau/2-1/2} a_n \as 0,
\end{align*}
since $\tau/2-1/2 < 0$ and since $\sup_{k \geq 2} k^{1-\tau/2} |U_{h_2,1:k}|/a_k < \infty$ with probability one as a consequence of Theorem~1 in \cite{DehWen10b}.

The proof under~(ii) is similar and follows by possibly letting $a = b \times 2 \gamma\delta / (3\gamma\delta + \delta + 5\gamma + 2)$.
\end{proof}

\section{A dependent multiplier central limit theorem}

Let $(Y_i)_{i \in \N}$ be a strictly stationary sequence of centered random variables. Furthermore, let $M > 0$ be a large integer, and let $(\xi_{i,n}^{(1)})_{i \in \N}$,\dots,$(\xi_{i,n}^{(M)})_{i \in \N}$ be independent copies of the same dependent multiplier sequence (see Definition~\ref{defn:mult_seq}). Then, for any $m \in \{1,\dots,M\}$ and $s \in [0,1]$, let 
$$
\Z_n(s) = \frac{1}{\sqrt{n}} \sum_{i=1}^{\ip{ns}} Y_i, \qquad  \Z_n^{(m)}(s) = \frac{1}{\sqrt{n}} \sum_{i=1}^{\ip{ns}}\xi_{i,n}^{(m)}  Y_i.
$$

\begin{prop}[Dependent multiplier central limit theorem]
\label{prop:func_mult}
Assume that $(Y_i)_{i \in \N}$ is a strictly stationary sequence of centered random variables with $(4+\delta)$-moments for some $\delta > 0$ and such that the strong mixing coefficients associated with $(Y_i)_{i \in \N}$ satisfy $\alpha_r = O(r^{-b})$, $b > 2(4+\delta)/\delta$. Then,
$$
\sigma^2 = \Ex(Y_0) + 2 \sum_{i=1}^\infty \Ex(Y_0Y_i) < \infty.
$$ 
Furthermore, let $(\xi_{i,n}^{(1)})_{i \in \N}$,\dots,$(\xi_{i,n}^{(M)})_{i \in \N}$ be independent copies of the same dependent multiplier sequence satisfying (M1)--(M3) in Definition~\ref{defn:mult_seq} such that $\ell_n = O(n^{1/2 - \eps})$ for some $1/(6+2\delta) < \eps < 1/2$. As a consequence, 
$$
( \Z_n, \Z_n^{(1)},\dots, \Z_n^{(M)} ) \leadsto ( \Z, \Z^{(1)},\dots,\Z^{(M)} )
$$
 in $\{\ell^\infty([0,1]) \}^{M+1}$, where $\Z = \sigma \B$ with $\B$ a standard Brownian motion, and where $\Z^{(1)},\dots,\Z^{(M)}$ are independent copies of $\Z$.
\end{prop}

The proof of the previous result is based on two lemmas, which are given first.

\begin{lem}[Finite-dimensional convergence]
\label{lem:fidi}
Assume that $(Y_i)_{i \in \N}$ is a strictly stationary sequence of centered random variables with $(4+\delta)$-moments for some $\delta > 0$ and such that the strong mixing coefficients associated with $(Y_i)_{i \in \N}$ satisfy $\alpha_r = O(r^{-b})$, $b > (4+\delta)(6+2\delta)/(2+\delta)^2$.  Also, let $(\xi_{i,n}^{(1)})_{i \in \N}$,\dots,$(\xi_{i,n}^{(M)})_{i \in \N}$ be independent copies of the same dependent multiplier sequence satisfying (M1)--(M3) in Definition~\ref{defn:mult_seq} such that $\ell_n = O(n^{1/2 - \eps})$ for some $1/(6+2\delta) < \eps < 1/2$. Then, the finite dimensional distributions of $( \Z_n, \Z_n^{(1)},\dots, \Z_n^{(M)} ) $ converge weakly to those of $( \Z, \Z^{(1)},\dots,\Z^{(M)} )$.
\end{lem}

\begin{proof}
The proof is an adapation of that of Lemma~A.1 in \cite{BucKoj14}. Fix $m \in \{1,\dots,M\}$. For the sake of brevity, we shall only show that the finite-dimensional distributions of $(\Z_n, \Z_n^{(m)})$ converge weakly to those of $(\Z, \Z^{(m)} )$, the proof of the stated result being a more notationally complex version of the proof of the latter result.

Let $q \in \N$, $q > 1$, be arbitrary, and let $(s_1,t_1),\dots,(s_q,t_q) \in [0,1]^2$. The result is proved if we show that
$$
\left(\Z_n(s_1),\Z_n^{(m)}(t_1),\dots, \Z_n(s_q),\Z_n^{(m)}(t_q)\right) \leadsto \left(\Z (s_1),\Z^{(m)} (t_1),\dots, \Z(s_q), \Z^{(m)}(t_q)\right).
$$
Let $c_1,d_1,\dots,c_q,d_q \in \R$ be arbitrary. By the Cram\'er--Wold device, it then suffices to show that
$$
Z_n = \sum_{l=1}^q c_l \Z_n (s_l) + \sum_{l=1}^q d_l \Z_n^{(m)} (t_l) \leadsto Z = \sum_{l=1}^q c_l \Z (s_l) + \sum_{l=1}^q d_l \Z^{(m)}(t_l).
$$
Now, for any $i \in \{1,\dots,n\}$, let 
$$
Z_{i,n} = \sum_{l=1}^q c_l Y_i \1(i \leq \ip{ns_l}) \qquad \mbox{and} \qquad
Z_{i,n}^{(m)} = \xi_{i,n}^{(m)} \sum_{l=1}^q d_l Y_i \1(i \leq \ip{nt_l}).
$$
Hence, $Z_n = n^{-1/2} \sum_{i=1}^n (Z_{i,n} + Z_{i,n}^{(m)})$. To prove the convergence in distribution of $Z_n$ to $Z$, we employ a blocking technique \citep[see, e.g.,][page 31]{DehPhi02}. Each block is composed of a big subblock followed by a small subblock. Let $1/(6+2\delta) < \eta_b < \eta_s < \eps$ such that $\eta_s < 1/2 - 1/a$, where $a = b (2+\delta)/(4+\delta) $. Notice that the condition on $b$ implies that $a > (6+2\delta)/(2+\delta)$, which is equivalent to $1/(6+2\delta) < 1/2 - 1/a$. Hence, it is possible to choose $\eta_b$ and $\eta_s$ according to the above constraints. The length of the small subblocks is $s_n = \ip{n^{1/2-\eta_s}}$ and the length of the big subblocks is $b_n = \ip{n^{1/2-\eta_b}}$ so that the length of a block is $b_n + s_n$. The total number of blocks is $k_n = \ip{n/(b_n+s_n)}$, and we can write $n = k_n (b_n + s_n) + \{n - k_n (b_n + s_n)\}$. Note that $s_n\sim n^{1/2-\eta_s}, b_n \sim n^{1/2-\eta_b}$ and $k_n\sim n^{1/2+\eta_b}$ and that both $b_n$ and $s_n$ dominate~$\ell_n$. As we continue, $n$ is taken sufficiently large so that $b_n > s_n > \ell_n$. Notice also that the condition $\eta_s < 1/2 - 1/a$ implies that $n s_n^{-a} \to 0$. Now, for any $j \in \{1,\dots,k_n\}$, let
$$
B_{j,n} = \sum_{i=(j-1)(b_n+s_n)+1}^{(j-1)(b_n+s_n) + b_n}  (Z_{i,n} + Z_{i,n}^{(m)}) \qquad \mbox{and} \qquad S_{j,n} = \sum_{i=(j-1)(b_n+s_n)+ b_n + 1}^{j(b_n+s_n)}  (Z_{i,n} + Z_{i,n}^{(m)}) 
$$
be the sums of the $ (Z_{i,n} + Z_{i,n}^{(m)})$ in the $j$th big subblock and the $j$th small subblock, respectively.
Then, 
$$
Z_n = n^{-1/2} \sum_{j=1}^{k_n} B_{j,n} + n^{-1/2} \sum_{j=1}^{k_n} S_{j,n} + n^{-1/2} R_n, 
$$
where $R_n = \sum_{i=k_n (b_n + s_n) + 1}^n  (Z_{i,n} + Z_{i,n}^{(m)})$ is the sum of the $ (Z_{i,n} + Z_{i,n}^{(m)})$ after the last small subblock. It follows that
\begin{multline}
\label{varZn}
\Var(Z_n) = \Var \left( n^{-1/2} \sum_{j=1}^{k_n} B_{j,n} \right)  + 2 n^{-1} \sum_{j,j'=1}^{k_n} \Ex ( B_{j,n} S_{j',n} ) + 2 n^{-1} \sum_{j=1}^{k_n} \Ex ( B_{j,n} R_n ) \\ + n^{-1} \sum_{j,j'=1}^{k_n} \Ex(S_{j,n}S_{j',n}) + 2 n^{-1} \sum_{j=1}^{k_n} \Ex ( S_{j,n} R_n) + \Ex ( n^{-1}  R_n^2 ).
\end{multline}
We shall now show that all the terms on the right except the first one tend to zero. Notice that the convergence of the fourth and sixth term to zero will imply that $| Z_n - n^{-1/2} \sum_{j=1}^{k_n} B_{j,n} | = | n^{-1/2} \sum_{j=1}^{k_n} S_{j,n} + n^{-1/2} R_n | \p 0$. We start with the second one. For any $i \in \N$, let $\gamma(i) = \Cov\{ Y_0, Y_i \}$. We have
$$
\Ex(B_{j,n}S_{j',n}) = \sum_{i=(j-1)(b_n+s_n) + 1}^{j(b_n+s_n) + b_n} \sum_{i'=(j'-1)(b_n+s_n)+ b_n + 1}^{j'(b_n+s_n)} \{ \Ex(Z_{i,n} Z_{i',n}) + \Ex (Z_{i,n}^{(m)} Z_{i',n}^{(m)}) \}.
$$
Now,
$$
| \Ex(Z_{i,n} Z_{i',n}) | \leq  \sum_{l,l'=1}^q | c_l c_{l'} | |\gamma(|i'-i|)| \leq 10 \alpha_{|i'-i|}^{(2+\delta)/(4+\delta)} \|Y_0\|_{4+\delta}^2 \sum_{l,l'=1}^q | c_l c_{l'} |,
$$
where the last inequality is a consequence of Lemma~3.11 in \cite{DehPhi02} with $r=s=4+\delta$ and $t = (4+\delta)/(2+\delta)$, which implies that 
\begin{equation}
\label{eq:covineq}
| \gamma(i) | \leq 10 \alpha_i^{(2+\delta)/(4+\delta)} \|Y_0\|_{4+\delta}^2, \qquad i \in \N. 
\end{equation}
Similarly,
$$
| \Ex(Z_{i,n}^{(m)} Z_{i',n}^{(m)}) | \leq \1( |i'-i|  \leq \ell_n) 10 \alpha_{|i'-i|}^{(2+\delta)/(4+\delta)} \|Y_0\|_{4+\delta}^2 \sum_{l,l'=1}^q | d_l d_{l'} |
$$
since, by Cauchy-Schwarz's inequality, $\Ex(\xi_{i,n}^{(m)} \xi_{i',n}^{(m)} ) \leq \Ex\{ (\xi_{0,n}^{(m)})^2 \} = 1$. It follows that
$$
|\Ex(B_{j,n}S_{j,n})| \leq \mbox{const} \times \sum_{i=1}^{b_n} \sum_{i'=b_n + 1}^{b_n + s_n} \alpha_{|i-i'|}^{(2+\delta)/(4+\delta)} \leq \mbox{const} \times  \sum_{i=1}^{b_n + s_n-1} i \alpha_ i^{(2+\delta)/(4+\delta)} < \infty
$$ 
since $\sum_{i=1}^\infty i \alpha_i^{(2+\delta)/(4+\delta)} < \infty$ as $a>2$. Similarly, we obtain that $|\Ex(B_{j,n}S_{j-1,n})| < \infty$. For $j' \ge j + 1$ or $j > j' + 1$, $\Ex(B_{j,n}S_{j',n}) = O(b_n s_n \alpha_{b_n}^{(2+\delta)/(4+\delta)}) = O(b_n s_n b_n^{-a} )$. Hence, 
$$
2 n^{-1} \sum_{j,j'=1}^{k_n} \Ex ( B_{j,n} S_{j',n} ) = O(n^{-1} k_n) + O(n^{-1} k_n^2 b_n s_n b_n^{-a} ) = O(b_n^{-1}) + O(n s_n b_n^{-a-1} ).
$$ 
Since $n s_n b_n^{-a-1} < n s_n^{-a}$, the previous term converges to zero. In a similar way, for the third summand on the right-hand side of~\eqref{varZn}, we have
\begin{multline*}
2 n^{-1} \sum_{j=1}^{k_n} \Ex ( B_{j,n} R_n ) = 2 n^{-1} \sum_{j=1}^{k_n - 1} \Ex ( B_{j,n} R_n ) + 2 n^{-1} \Ex ( B_{k_n,n} R_n ) \\= O(n^{-1} k_n b_n (n-k_n(b_n+s_n)) \alpha_{b_n}^{(2+\delta)/(4+\delta)}) + O(n^{-1} b_n (n-k_n(b_n+s_n))) \\ =  O(b_n^{-a+1}) + O(n^{-1} b_n^2) \to 0
\end{multline*}
using the fact that $n-k_n(b_n+s_n) < b_n +s_n$. The case of the fifth summand is similar. Regarding the fourth summand in \eqref{varZn}, we have 
$$
\Ex(S_{j,n}S_{j',n}) = \sum_{i=(j-1)(b_n+s_n) + b_n + 1}^{j(b_n+s_n) } \sum_{i'=(j'-1)(b_n+s_n)+ b_n + 1}^{j'(b_n+s_n)} \{ \Ex(Z_{i,n} Z_{i',n}) + \Ex (Z_{i,n}^{(m)} Z_{i',n}^{(m)}) \},
$$
which implies that
\begin{multline*}
\Ex(S_{j,n}^2) \le \mbox{const} \times \sum_{i,i'=b_n+1}^{b_n+s_n} \alpha_{|i'-i|}^{(2+\delta)/(4+\delta)} \leq \mbox{const} \times 2 \sum_{i=0}^{s_n - 1} (s_n - i) \alpha_i^{(2+\delta)/(4+\delta)} \\ \leq \mbox{const} \times s_n \sum_{i=0}^\infty \alpha_i^{(2+\delta)/(4+\delta)} = O(s_n)
\end{multline*}
and that, for $j\ne j'$, $\Ex(S_{j,n}S_{j',n}) = O(s_n^2 \alpha_{b_n}^{(2+\delta)/(4+\delta)}) = O(s_n^2 b_n^{-a})$. Hence,
\[
n^{-1} \sum_{j,j'=1}^{k_n} \Ex(S_{j,n}S_{j',n}) = O(n^{-1}k_ns_n) + O(n^{-1}k_n^2s_n^2b_n^{-a}) = O(b_n^{-1} s_n) + O(n s_n^2b_n^{-a-2}) 
\]
which converges to $0$ since $b_n^{-1} s_n \to 0$ and $n s_n^2b_n^{-a-2} < n s_n^{-a}$. Finally, for the sixth summand in \eqref{varZn}, we have
\[
\Ex ( n^{-1}  R_n^2 ) =O(n^{-1}\{ n-k_n(b_n+s_n)\}^2) = O(n^{-1} (b_n +s_n)^2) = O(n^{-1} b_n^2) 
\]
since $n-k_n(b_n+s_n) < b_n +s_n$.

In order to prove that $Z_n$ converges in distribution to $Z$, it suffices therefore to prove that $n^{-1/2} \sum_{j=1}^{k_n} B_{j,n}$ converges in distribution to $Z$. Let $\psi_{j,n}(t) = \exp(itn^{-1/2} B_{j,n})$, $t \in \R$, $j \in \{1,\dots,k_n\}$, and observe that the characteristic function of $n^{-1/2} \sum_{j=1}^{k_n} B_{j,n}$ can be written as $t \mapsto \Ex \left\{ \prod_{j=1}^{k_n} \psi_{j,n}(t) \right\}$. Also, for two $\sigma$-fields $\FF_1$ and $\FF_2$, let
$$
\alpha(\FF_1,\FF_2) = \sup_{A \in \FF_1, B \in \FF_2} | \Pr(A \cap B) - \Pr(A) \Pr(B) | .
$$
Now, for any $t \in \R$,  we can write
\begin{multline*}
\left| \Ex\left\{ \prod_{j=1}^{k_n} \psi_{j,n}(t) \right\} - \prod_{j=1}^{k_n} \Ex \{ \psi_{j,n}(t) \} \right| \leq \left| \Ex\left\{ \prod_{j=1}^{k_n} \psi_{j,n}(t) \right\} - \Ex \{ \psi_{1,n}(t) \} \Ex\left\{ \prod_{j=2}^{k_n} \psi_{j,n}(t) \right\} \right| \\ +  \left| \Ex \{ \psi_{1,n}(t) \} \right| \left| \Ex\left\{ \prod_{j=2}^{k_n} \psi_{j,n}(t) \right\} - \Ex \{ \psi_{2,n}(t) \} \Ex\left\{ \prod_{j=3}^{k_n} \psi_{j,n}(t) \right\} \right| + \dots \\ \dots + \left| \prod_{j=1}^{k_n-2} \Ex \{ \psi_{j,n}(t) \} \right| \left| \Ex \left\{ \prod_{j=k_n - 1}^{k_n}  \psi_{j,n}(t) \right\} - \prod_{j=k_n - 1}^{k_n} \Ex \{ \psi_{j,n}(t) \} \right|.
\end{multline*}
Using the fact that the modulus of a characteristic function is smaller than one and applying $k_n-1$ times Lemma~3.9 of \cite{DehPhi02}, we obtain 
$$
\left| \Ex\left\{ \prod_{j=1}^{k_n} \psi_{j,n}(t) \right\} - \prod_{j=1}^{k_n} \Ex \{ \psi_{j,n}(t) \} \right| \leq 2 \pi k_n \max_{1 \leq i \leq k_n-1} \alpha \left[ \sigma \left\{ \psi_{i,n}(t) \right\}, \sigma \left\{ \prod_{j=i+1}^{k_n}  \psi_{j,n}(t) \right\} \right].
$$
Since the big subblocks are $s_n$ observations apart, the right-hand side of the previous inequality is smaller than $2 \pi k_n \alpha_{s_n} = O(k_n s_n^{-a})$ which tends to zero as $k_n s_n^{-a} \leq n s_n^{-a} \to 0$. Hence, for any $t \in \R$, 
$$
\left| \Ex\left\{ \prod_{j=1}^{k_n} \psi_{j,n}(t) \right\} - \prod_{j=1}^{k_n} \Ex \{ \psi_{j,n}(t) \} \right| \to 0.
$$
In other words, the characteristic function of $n^{-1/2} \sum_{j=1}^{k_n} B_{j,n}$ is asymptotically equivalent to the characteristic function of $n^{-1/2} \sum_{j=1}^{k_n} B_{j,n}'$, where $B_{1,n}',\dots,B_{k_n,n}'$ are independent and $B_{j,n}'$ and $B_{j,n}$ have the same distribution for all $j \in \{1,\dots,k_n\}$. To conclude that $n^{-1/2} \sum_{j=1}^{k_n} B_{j,n}$ converges in distribution to $Z$, it suffices therefore to show that $n^{-1/2} \sum_{j=1}^{k_n} B_{j,n}'$ converges in distribution to $Z$. This will be accomplished using the Lindeberg--Feller central limit theorem for triangular arrays. Hence, let us first show that $\Var \left(n^{-1/2} \sum_{j=1}^{k_n} B_{j,n}' \right) \to \Var(Z)$. 

We have 
$$
\Var(Z) = \sum_{l,l'=1}^q c_l c_{l'} (s_l \wedge s_{l'}) \sum_{i \in \Z} \gamma(|i|) + \sum_{l,l'=1}^q d_l d_{l'} (t_l \wedge t_{l'}) \sum_{i \in \Z} \gamma(|i|).
$$
Note that $\sum_{i \in \Z} \gamma(|i|) = \Var\{\Z(1) \} < \infty$ since, from~\eqref{eq:covineq}, 
\begin{equation}
\label{eq:absconv}
\sum_{i \in \Z} |\gamma(|i|)| = |\gamma(0)| + 2\sum_{i=1}^\infty |\gamma(i)|  \leq \Var(Y_0) + 20 \sum_{i=1}^\infty \alpha_i^{(2+\delta)/(4+\delta)} < \infty.
\end{equation}
Now, we shall first show that
$$
\Var \left(n^{-1/2} \sum_{j=1}^{k_n} B_{j,n}' \right) = \Var(Z_n) + o(1)
$$
and then that $\Var(Z_n) \to \Var(Z)$. We have
\begin{align*}
\Var \left(n^{-1/2} \sum_{j=1}^{k_n} B_{j,n}' \right) &= n^{-1} \sum_{j=1}^{k_n} \Var \left( B_{j,n}'\right) = n^{-1} \sum_{j=1}^{k_n} \Var \left( B_{j,n}\right) \\
&= \Var \left( n^{-1/2} \sum_{j=1}^{k_n} B_{j,n} \right)  - n^{-1} \sum_{j,j'=1 \atop j \ne j'}^{k_n} \Ex\left( B_{j,n} B_{j',n} \right).
\end{align*}
From \eqref{varZn}, we know that $\Var ( n^{-1/2} \sum_{j=1}^{k_n} B_{j,n} ) = \Var(Z_n) + o(1)$. Hence, it remains to show that the double sum in the last displayed formula converges to $0$. Proceeding as for the summands on the right of~\eqref{varZn}, we have that, for $j\ne j'$,  $\Ex( B_{j,n} B_{j',n} ) = O(b_n^2\alpha_{s_n}^{(2+\delta)/(4+\delta)}) = O(b_n^2 s_n^{-a})$. Hence,
\[
n^{-1} \sum_{j,j'=1 \atop j \ne j'}^{k_n} \Ex \left( B_{j,n} B_{j',n} \right) = O(n^{-1} k_n^2 b_n^2 s_n^{-a}) = O(n s_n^{-a}) \to 0.
\]
Thus, it remains to show that $\Var(Z_n) \to \Var(Z)$. Now, 
$$
\Var(Z_n) = n^{-1} \sum_{i,i'=1}^n \{ \Ex(Z_{i,n} Z_{i',n}) + \Ex (Z_{i,n}^{(m)} Z_{i',n}^{(m)}) \}. 
$$
It follows that 
\begin{multline}
\label{varZn2}
\Var(Z_n) = \sum_{l,l'=1}^q c_l c_{l'} n^{-1} \sum_{i=1}^{\ip{ns_l}} \sum_{i'=1}^{\ip{ns_{l'}}} \gamma(|i'-i|)   \\
+ \sum_{l,l'=1}^q d_l d_{l'} n^{-1} \sum_{i=1}^{\ip{nt_l}} \sum_{i'=1}^{\ip{nt_{l'}}} \varphi\{(i'-i)/\ell_n\} \gamma(|i'-i|),
\end{multline}
where $\varphi$ is the function appearing in Assumption~(M3). Let us first deal with the second term on the right. Let $l,l'\in \{1,\dots, q\}$ be arbitrary and suppose without loss of generality that $t_l \le t_{l'}$. Then,
\begin{multline} 
\label{eq:decomp1}
n^{-1} \sum_{i=1}^{\ip{nt_l}} \sum_{i'=1}^{\ip{nt_{l'}}} \varphi\{(i'-i)/\ell_n\} \gamma(|i'-i|)  =  n^{-1} \sum_{i=1}^{\ip{nt_l}} \sum_{i'=1}^{\ip{nt_l}}  \varphi\{(i'-i)/\ell_n\} \gamma(|i'-i|)    \\
+ n^{-1} \sum_{i=1}^{\ip{nt_l}} \sum_{i'= \ip{nt_l}+1}^{\ip{nt_{l'}}} \varphi\{(i'-i)/\ell_n\} \gamma(|i'-i|).   
\end{multline}
The first sum on the right-hand side is equal to 
$$
n^{-1} \sum_{i=-\ip{nt_l}}^{\ip{nt_l}}  \{\ip{nt_l}-|i|\} \varphi(i/\ell_n) \gamma(|i|) = \sum_{i=-\ip{nt_l}}^{\ip{nt_l}}  \{\lambda_n(0,t_l)-|i|/n\} \varphi(i/\ell_n) \gamma(|i|)
$$
and converges to $t_l \sum_{i\in\Z} \gamma(|i|)$ by Assumption~(M3),~\eqref{eq:absconv} and dominated convergence. The second sum on the right-hand side of~\eqref{eq:decomp1} is bounded in absolute value by $\mbox{const} \times n^{-1} \sum_{i=1}^{\ip{nt_{l'}}-1} i \alpha_i^{(2+\delta)/(4+\delta)} \to 0$. Hence, the second term on the right of~\eqref{varZn2} converges to $\sum_{l,l'=1}^q d_l d_{l'} (t_l \wedge t_{l'}) \sum_{i \in \Z} \gamma(|i|)$. Similarly, the first term on the right of~\eqref{varZn2} converges to $\sum_{l,l'=1}^q c_l c_{l'} (s_l \wedge s_{l'}) \sum_{i \in \Z} \gamma(|i|)$. Thus, $\Var(Z_n) \to \Var(Z)$.

To be able to conclude that $n^{-1/2} \sum_{j=1}^{k_n} B_{j,n}'$ converges in distribution to $Z$, it remains to prove the Lindeberg condition of the Lindeberg-Feller theorem, i.e., that, for every $\epsilon > 0$, 
$$
n^{-1} \sum_{j=1}^{k_n} \Ex \{ (B_{j,n}')^2  \1( |B_{j,n}'| > n^{1/2} \epsilon ) \}  
= n^{-1} \sum_{j=1}^{k_n} \Ex \{ B_{j,n}^2  \1( |B_{j,n}| > n^{1/2} \epsilon ) \} \to 0.
$$
Let $\epsilon > 0$ be arbitrary. Using H\"older's inequality with $p=1+\nu/2$, where $0 < \nu \leq 2 + \delta$ is to be chosen later on, and Markov's inequality, we have 
\begin{align*}
n^{-1} \sum_{j=1}^{k_n} &\Ex \{ B_{j,n}^2  \1( |B_{j,n}| > n^{1/2} \epsilon ) \} \\ 
&\leq n^{-1} \sum_{j=1}^{k_n} \{ \Ex ( |B_{j,n}|^{2+\nu} ) \}^{2/(2+\nu)} \{ \Pr( |B_{j,n}| > n^{1/2} \epsilon )\}^{\nu/(2+\nu)} \\
&\leq n^{-1} \sum_{j=1}^{k_n} \{ \Ex ( |B_{j,n}|^{2+\nu} ) \}^{2/(2+\nu)} \{ \Pr( |B_{j,n}|^{2+\nu} > n^{(2+\nu)/2} \epsilon^{2+\nu} )\}^{\nu/(2+\nu)} \\
&\leq n^{-1} \sum_{j=1}^{k_n} \{ \Ex ( |B_{j,n}|^{2+\nu} ) \}^{2/(2+\nu)} \{ \Ex( |B_{j,n}|^{2+\nu} ) \}^{\nu/(2+\nu)}  (n^{1/2} \epsilon )^{-\nu} \\
&\leq n^{-1} \sum_{j=1}^{k_n} \Ex ( |B_{j,n}|^{2+\nu} )  n^{-\nu/2} \epsilon^{-\nu}.
\end{align*}
Now, from Minkowski's inequality,
$$
\{ \Ex ( |B_{j,n}|^{2+\nu} ) \}^{1/(2+\nu)} \leq \sum_{i=(j-1)(b_n+s_n)+1}^{(j-1)(b_n+s_n) + b_n} \left[ \{ \Ex ( |Z_{i,n}|^{2+\nu} ) \}^{1/(2+\nu)} + \{ \Ex ( |Z_{i,n}^{(m)}|^{2+\nu} ) \}^{1/(2+\nu)} \right]= O(b_n)
$$
since, for any $\nu \in(0, 2+\delta)$, 
$$
\max_{1 \leq i \leq n} \Ex ( |Z_{i,n}|^{2+\nu} )  \leq \Ex \left( \left[ \sum_{l=1}^q |c_l| |Y_0| \right]^{2+\nu} \right) < \infty,
$$ 
and
$$
\max_{1 \leq i \leq n} \Ex ( |Z_{i,n}^{(m)}|^{2+\nu} )  \leq \Ex \left( \left[ | \xi_{0,n} | \sum_{l=1}^q |d_l| | Y_0 | \right]^{2+\nu} \right) < \infty.
$$ 
It follows that 
$$
n^{-1} \sum_{j=1}^{k_n} \Ex \{ B_{j,n}'^2  \1( |B_{j,n}'| > n^{1/2} \epsilon ) \} = O(n^{-1} k_n b_n^{2+\nu} n^{-\nu/2}) = O(b_n^{1+\nu} n^{-\nu/2}) = O(n^{1/2 - \eta_b (1+\nu)}),
$$
which converges to zero for $\nu > 1/(2 \eta_b) - 1$. The condition $1/(6+2\delta) < \eta_b$ imposed at the beginning of the proof is equivalent to $2+\delta > 1/(2 \eta_b) - 1$. Hence, the desired convergence is obtained by taking $\nu = 2 + \delta$.
\end{proof}

\begin{lem}[Moment inequality]
\label{lem:moments4}
Assume that $(Y_i)_{i \in \N}$ is a strictly stationary sequence of centered random variables with $(4+\delta)$-moments for some $\delta > 0$ and such that the strong mixing coefficients associated with $(Y_i)_{i \in \N}$ satisfy $\alpha_r = O(r^{-b})$, $b > 2(4+\delta)/\delta$.  Also, for any $m \in \{1,\dots,M\}$, let $(\xi_{i,n}^{(m)})_{i \in \N}$ be a sequence satisfying~(M1) in Definition~\ref{defn:mult_seq}. Then, for any $0 \leq s \leq t \leq 1$,
$$ 
\Ex \left[ \{ \Z_n^{(m)}(s) - \Z_n^{(m)}(t) \}^4 \right]  \leq \kappa \{ \lambda_{n}(s,t) \}^2,
$$
where $\kappa > 0$ is a constant depending on the mixing coefficients, $\Ex[\{\xi_{0,n}^{(m)}\}^4]$ and $\|Y_0\| _{4+\delta}^4$. 
\end{lem}

\begin{proof}
The proof is an adapation of that of Lemma~A.2 in \cite{BucKoj14}.
The result holds trivially if $\ip{ns} = \ip{nt}$. Let us therefore assume that $\ip{nt} - \ip{ns} \geq 1$. Proceeding as in the proof of Lemma 3.22 in \cite{DehPhi02}, we can write
\begin{align} 
\nonumber
& \Ex \left[ \{ \Z_n^{(m)}(s) - \Z_n^{(m)}(t) \}^4 \right]  \\
	& \hspace{.5cm}= \frac{1}{n^2} \sum_{i_1, i_2, i_3, i_4=\ip{ns}+1}^{\ip{nt}} \Ex[\xi_{i_1,n}^{(m)}\xi_{i_2,n}^{(m)} \xi_{i_3,n}^{(m)} \xi_{i_4,n}^{(m)}] \Ex[Y_{i_1}Y_{i_2} Y_{i_3}Y_{i_4}],  \nonumber \\
\label{eq:sum1}
	&\hspace{.5cm} \le \frac{4!\lambda_n(s,t)}{ n} \sum_{0\le i,j,k \le \ip{nt}-\ip{ns}-1 \atop i+j+k \le \ip{nt}-\ip{ns}-1} | \Ex[ \xi_{0,n}^{(m)}\xi_{i,n}^{(m)} \xi_{i+j,n}^{(m)} \xi_{i+j+k,n}^{(m)}] \Ex[Y_{0}Y_{i} Y_{i+j}Y_{i+j+k}] |.
\end{align} 
On the one hand, $|\Ex[\xi_{0,n}^{(m)}\xi_{i,n}^{(m)} \xi_{i+j,n}^{(m)} \xi_{i+j+k,n}^{(m)}] |\le \Ex[\{ \xi_{0,n}^{(m)}\}^4]$. On the other hand, by Lemma~3.11 of \cite{DehPhi02} and using the generalized verson of Holder's inequality, we have
\begin{align*}
 \Ex[Y_{0} (Y_{i} Y_{i+j}Y_{i+j+k}) ] &\le 10 \alpha_i^{\delta/(4+\delta)} \| Y_0\|_{4+\delta} \| Y_iY_{i+j}Y_{i+j+k} \|_{(4+\delta)/3} \le 10 \alpha_i^{\delta/(4+\delta)} \| Y_0\|_{4+\delta}^4, \\
\Ex[(Y_{0} Y_{i} Y_{i+j})Y_{i+j+k} ] &\le 10 \alpha_k^{\delta/(4+\delta)} \| Y_0\|_{4+\delta}^4,
\end{align*}
and
\begin{align*}
|\Ex[(Y_{0} Y_{i})( Y_{i+j}Y_{i+j+k}) ] | &\le | \Ex[Y_{0} Y_{i} ] \Ex[ Y_{i+j}Y_{i+j+k} ] |  
+ 10 \alpha_j^{\delta/(4+\delta)} \|Y_0Y_i\|_{(4+\delta)/2} \| Y_{i+j}Y_{i+j+k} \|_{(4+\delta)/2}  \\
 &\le 100 \alpha_i^{(2+\delta)/(4+\delta)}  \alpha_k^{(2+\delta)/(4+\delta)} \| Y_0\|_{4+\delta}^4 + 10 \alpha_j^{\delta/(4+\delta)} \| Y_0\|_{(4+\delta)}^4.
\end{align*}
Proceeding as in Lemma 3.22 of \cite{DehPhi02}, we split the sum on the right of \eqref{eq:sum1} into three sums according to which of the indices $i,j,k$ is the largest. Combining this decomposition with the three previous inequalities, we obtain
\begin{multline*}
\Ex \left[ \{ \Z_n^{(m)}(s) - \Z_n^{(m)}(t) \}^4 \right]  
 \le \frac{24 \Ex[\{\xi_{0,n}^{(m)}\}^4]\|Y_0\| _{4+\delta}^4 \lambda_n(s,t)}{n} \\ \times \left\{ 100  \sum_{j=0}^{\ip{nt}-\ip{ns}-1} \sum_{i,k \le j} \alpha_i^{(2+\delta)/(4+\delta)} \alpha_k^{(2+\delta)/(4+\delta)} + 30 \sum_{i=0}^{\ip{nt}-\ip{ns}-1} \sum_{j,k\le i} \alpha_i^{\delta/(4+\delta)}  \right\}.
\end{multline*}
From the condition on the mixing coefficients, there exists a constant $C > 0$ such that
$$
 \sum_{i=0}^{\ip{nt}-\ip{ns}-1} \sum_{j,k\le i} \alpha_i^{\delta/(4+\delta)} \leq \sum_{i=0}^{\ip{nt}-\ip{ns}-1} i^2 \alpha_i^{\delta/(4+\delta)} \leq C \sum_{i=1}^{\ip{nt}-\ip{ns}} i^{2-a},
$$
where $a = b \delta / (4+\delta) > 2$. From well-known results on Riemann series, if $-1 < 2-a < 0$, the sum on the right of the previous inequality is $O((\ip{nt}-\ip{ns})^{3-a})$. Hence, $\sum_{i=0}^{\ip{nt}-\ip{ns}-1} \sum_{j,k\le i} \alpha_i^{\delta/(4+\delta)} = O((\ip{nt}-\ip{ns})^\tau)$, where $\tau = \max(3-a,0) < 1$. Also, since $a > 2$, $\sum_{i=1}^\infty \alpha_i^{(2+\delta)/(4+\delta)} < \infty$, and therefore
$$
\Ex \left[ \{ \Z_n^{(m)}(s) - \Z_n^{(m)}(t) \}^4 \right]  \leq K \lambda_{n}(s,t) [ \lambda_{n}(s,t)  + n^{\tau-1} \{\lambda_{n}(s,t)\}^\tau],
$$
where $K > 0$ is a constant depending on the mixing coefficients, $\Ex[\{\xi_{0,n}^{(m)}\}^4]$ and $\|Y_0\| _{4+\delta}^4$. The latter inequality implies that
$$
\Ex \left[ \{ \Z_n^{(m)}(s) - \Z_n^{(m)}(t) \}^4 \right] \leq K \{ \lambda_{n}(s,t) \}^2 [ 1 + n^{\tau-1} \{ \lambda_{n}(s,t) \}^{\tau-1}] \leq K \{ \lambda_{n}(s,t) \}^2 (1 + 1),
$$
where we have used the fact that $1/n \leq \lambda_{n}(s,t)$ and that $\tau-1 <  0$.
\end{proof}

\begin{proof}[\bf Proof of Proposition~\ref{prop:func_mult}] 
Weak convergence of the finite-dimension\-al distributions is established in Lemma~\ref{lem:fidi}. Asymptotic tightness of $\Z_n$ is a consequence of the weak convergence of $\Z_n$ to $\Z$ in $\ell^\infty([0,1])$ proved in \citet[Theorem~2]{OodYos72}. Fix $m \in \{1,\dots,M\}$. To show asymptotic tightness of $\Z_n^{(m)}$, we shall first prove that $\Z_n^{(m)}$ is asymptotically uniformly equicontinuous in probability using Lemma~\ref{lem:moments4} together with Lemma~2 of \cite{BalDup80} \citep[see also][Theorem 3 and the remarks on page 1665]{BicWic71}.

Let $T_n = \{i/n : i = 0,\dots,n \}$ and let $(s,t]$ be a non-empty set of $[0,1]$ whose boundary points lie in $T_n$. Also, let $\mu$ be the Lebesgue measure on $[0,1]$. By Markov's inequality and Lemma~\ref{lem:moments4}, we then have that, for any $\eps > 0$,
\begin{multline}
\label{eq:prob_bound}
\Pr( |\Z_n^{(m)}(s) - \Z_n^{(m)}(t)| \geq \eps) \leq \eps^{-4}  \Ex \left[ \{ \Z_n^{(m)}(s) - \Z_n^{(m)}(t) \}^4 \right]  \\ \leq \eps^{-4} \kappa \{ \lambda_{n}(s,t) \}^2 = \eps^{-4} \kappa \{ \mu((s,t]) \}^2.
\end{multline}
Let $\tilde \mu_n$ denote a finite measure on $T_n$ defined from its values on the singletons $\{s\}$ of $T_n$ as
$$
\tilde \mu_n(\{s\}) = \begin{cases} 0 & \text{ if } s  = 0, \\
 \mu( (s', s] ) &\text{ otherwise,}
	 \end{cases}
$$
where  $s'=\max\{ t \in T_n : t < s\}$. By additivity of $\tilde \mu_n$,~\eqref{eq:prob_bound} can be rewritten as
$$
\Pr( |\Z_n^{(m)}(s) - \Z_n^{(m)}(t)| \geq \eps) \leq \eps^{-4} \kappa \{ \tilde \mu_n((s,t] \cap T_n) \}^2.
$$
Next, consider a positive sequence $\delta_n \downarrow 0$, and let $\delta_n' \downarrow 0$ such that, for any $n \in \N$, $\delta_n' \in \{1/i : i \in \N\}$ and $\delta_n' \ge \max(\delta_n, 1/n)$. Then, for any $\eps > 0$,
$$
\Pr \left\{ \sup_{s,t \in [0,1] \atop |s - t| \leq \delta_n} |\Z_n^{(m)}(s) -\Z_n^{(m)}(t)| > \eps \right\} \leq \Pr \left\{ \sup_{s,t \in [0,1] \atop |s - t| \leq \delta_n'} |\Z_n^{(m)}(s) -\Z_n^{(m)}(t)| > \eps \right\}.
$$
Applying Lemma 2 of \cite{BalDup80}, we obtain that there exists a constant $C > 0$ depending on $\eps$ such that the probability on the right of the previous display is smaller than
$$
C \tilde \mu_n(T_n) \times  \sup_{s,t \in T_n \atop |s - t| \leq 3 \delta_n'} | \tilde \mu_n( \{0,...,s\}) - \tilde \mu_n(\{0,...,t\}) |  \leq C 3\delta_n' \to 0.
$$
Hence, $\Z_n^{(m)}$ is asymptotically uniformly equicontinuous in probability and therefore asymptotically tight. 
The proof is complete as marginal asymptotic tightness implies joint asymptotic tightness.
\end{proof}

\section{Proof of Proposition~\ref{prop:multUn}}
\label{proof:prop:multUn}

\begin{proof}[\bf Proof of Proposition~\ref{prop:multUn}] 
For any $m \in \{1,\dots,M\}$, let
$$
\tilde \U_n^{(m)} (s) = \frac{2}{\sqrt{n}} \sum_{i=1}^{\ip{ns}} \xi_{i,n}^{(m)} \tilde h_{1,1:\ip{ns}} (\vec X_i), \qquad s \in [0,1],
$$
where, for $1\le k \le l \le n$,
\begin{equation}
\label{eq:tildeh1kl}
\tilde h_{1,k:l}(\vec X_i) = \frac{1}{l-k} \sum_{j=k \atop j \neq i}^l h(\vec X_i, \vec X_j) - \theta, \qquad i \in \{k,\dots,l\},
\end{equation}
with the convention that $\tilde h_{1,k:l} = 0$ if $k=l$. Fix $m \in \{1,\dots,M\}$ and let us first show~\eqref{eq:asym_equiv_checkUnm}. Proceeding as for the term (B.8) in \cite{BucKojRohSeg14}, it can be verified that
$$
\sup_{s \in [0,1]} | \check \U_n^{(m)} (s) - \tilde \U_n^{(m)} (s) | = \sup_{s \in [0,1]} \left| \U_n(s) \times \frac{2}{n} \sum_{i=1}^{\ip{ns}}  \xi_{i,n}^{(m)} \right| = o_\Pr(1).
$$
Hence, to show~\eqref{eq:asym_equiv_checkUnm}, it is enough to prove that
\begin{equation}
\label{eq:asym_equiv_tildeUnm}
\sup_{s \in [0,1]} \left| \tilde \U_n^{(m)}(s) - \frac{2}{\sqrt{n}} \sum_{i=1}^{\ip{ns}} \xi_{i,n}^{(m)} h_1(\vec X_i) \right| = o_\Pr(1).
\end{equation}
It is easy to verify that the result holds if the above supremum is restricted to $s \in [0,2/n)$. For any $s \in [2/n,1]$, using~\eqref{eq:h2}, we obtain 
\begin{align*}
\tilde \U_n^{(m)} (s) 
&= 
\frac{2}{\sqrt n (\ip{ns}-1) } \sum_{i,j=1 \atop i \neq j}^{\ip{ns}}  \xi_{i,n}^{(m)} \{ h_1(\vec X_i) + h_1(\vec X_j) + h_2(\vec X_i, \vec X_j) \}  \\
&= \frac{2}{\sqrt{n}}  \sum_{i=1}^{\ip{ns}}  \xi_{i,n}^{(m)} h_1(\vec X_i) + \V_n^{(m)}(s) + \W_n^{(m)}(s),
\end{align*}
where 
\begin{align}
\label{eq:Vnm}
\V_n^{(m)}(s) &= \frac{2}{\sqrt n(\ip{ns}-1) } \sum_{i,j=1 \atop i \neq j}^{\ip{ns}}  \xi_{i,n}^{(m)} h_1(\vec X_j) , \\
\label{eq:Wnm}
\W_n^{(m)}(s) &= \frac{2}{\sqrt n(\ip{ns}-1) }  \sum_{i,j=1 \atop i \neq j}^{\ip{ns}}  \xi_{i,n}^{(m)} h_2(\vec X_i, \vec X_j) .
\end{align}
To prove~\eqref{eq:asym_equiv_tildeUnm}, it remains therefore to show that both $\sup_{s\in[2/n,1]} | \V_n^{(m)}(s) | = o_\Pr(1)$ and $\sup_{s\in[2/n,1]} | \W_n^{(m)}(s) | = o_\Pr(1)$. First, for any $s\in [2/n,1]$, we write $\V_n^{(m)}(s) = \V_{n,1}^{(m)}(s) - \V_{n,2}^{(m)}(s)$, where
\begin{align*}
\V_{n,1}^{(m)}(s) 
&= 
\frac{1}{\ip{ns}-1}\sum_{i=1}^{\ip{ns}} \xi_{i,n}^{(m)} \times \frac{2}{\sqrt n}\sum_{j=1}^{\ip{ns}} h_1(\vec X_j), \\
\V_{n,2}^{(m)}(s) 
&=
\frac{1}{(\ip{ns}-1) } \times \frac{2}{\sqrt n} \sum_{i=1}^{\ip{ns}} \xi_{i,n}^{(m)} h_1(\vec X_i).
\end{align*}

Using the method of proof used for the term (B.8) in \cite{BucKojRohSeg14}, it can be verified that $\sup_{s\in[2/n,1]} |\V_{n,1}^{(m)}(s) | =o_\Pr(1)$. Furthermore, using the fact that $s \mapsto n^{-1/2} \sum_{i=1}^{\ip{ns}} \xi_{i,n}^{(m)} h_1(\vec X_i) \in \ell^\infty([0,1])$ is asymptotically equicontinuous in probability as a consequence of Proposition~\ref{prop:func_mult}, we have that 
$$
\sup_{s\in [2/n, n^{-1/2}]} | \V_{n,2}^{(m)}(s)  | \leq 2 \sup_{s \in [0,1] \atop s \leq n^{-1/2}} \left| \frac{1}{\sqrt{n}} \sum_{i=1}^{\ip{ns}} \xi_{i,n}^{(m)} h_1(\vec X_i)\right| = o_\Pr(1).
$$ 
Moreover, using again the weak convergence of $s \mapsto n^{-1/2} \sum_{i=1}^{\ip{ns}} \xi_{i,n}^{(m)} h_1(\vec X_i)$ in $\ell^\infty([0,1])$,
\[
\sup_{s \in [n^{-1/2}, 1] } \left| \V_{n,2}^{(m)}(s) \right |  
\le 
\frac{1}{\sqrt n -2} \times \sup_{s \in [0,1] } \left | \frac{2}{\sqrt n} \sum_{i=1}^{\ip{ns}} \xi_{i,n}^{(m)} h_1(\vec X_i) \right|  =O_\Pr(n^{-1/2}) =o_\Pr(1).
\]
Altogether, $\V_n^{(m)}$ converges uniformly to zero in probability as it was to be shown.

The uniform convergence in probability to zero of $\W_n^{(m)}$ in~\eqref{eq:Wnm} follows from Lemma~\ref{lem:Wn_p_0} below. Indeed, under~(i), we have $\sum_{r=1} r \beta_r^{\delta/(2+\delta)} \le \mbox{const} \times \sum_{r=1}^n r^{1-a} = O(1)$, where $a = b \delta/(2+\delta) > b \delta/(4+\delta) > 2$. Under~(ii), from Lemma~4.5 of \cite{DehWen10b}, we know that since $h$ satisfies the $\Pr$-Lipschitz-continuity or the variation condition, so does $h_2$ and, in addition, $\sum_{r=1}^n r \alpha_r^{\delta/(2+\delta)} \leq \mbox{const} \times \sum_{r=1}^n r^{1-a} = O(1)$ with $a = b \min\{\delta/(2+\delta), 2\gamma\delta/(3\gamma\delta+\delta+5\gamma+2)\} > 2$. This completes the proof of~\eqref{eq:asym_equiv_checkUnm}. 

Let us now prove~\eqref{eq:asym_equiv_hatUnm}. Using~\eqref{eq:hath1kl} and then~\eqref{eq:h2}, we obtain
\begin{multline*}
\sup_{s \in [0,1]} \left| \hat \U_n^{(m)}(s) - \frac{2}{\sqrt{n}} \sum_{i=1}^{\ip{ns}} \xi_{i,n}^{(m)} h_1(\vec X_i) \right| \leq \sup_{s \in [0,1]} \left| \frac{2}{\sqrt{n} (n-1)} \sum_{i=1}^{\ip{ns}} \xi_{i,n}^{(m)} \sum_{j=1 \atop j \neq i}^n h_1(\vec X_j) \right| \\ 
+ \sup_{s \in [0,1]} \left| \frac{2}{\sqrt{n} (n-1)} \sum_{i=1}^{\ip{ns}} \xi_{i,n}^{(m)} \sum_{j=1 \atop j \neq i}^n h_2(\vec X_i, \vec X_j) \right| + | \U_n(1) | \times \sup_{s \in [0,1]} \left| \frac{2}{n} \sum_{i=1}^{\ip{ns}}  \xi_{i,n}^{(m)} \right|.
\end{multline*}
The third term on the right of the previous inequality converges in probability to zero since $\U_n(1)$ converges weakly as a consequence of Proposition~\ref{prop:weakUn} and $\sup_{s \in [0,1]} \left| 2n^{-1} \sum_{i=1}^{\ip{ns}}  \xi_{i,n}^{(m)} \right| = o_\Pr(1)$, which can be shown by proceeding for instance as for the term $K_n$ in the proof of Lemma D.2 of \cite{BucKoj14}. To show that the first supremum on the right is $o_\Pr(1)$, one can use a decomposition similar to that used for~\eqref{eq:Vnm} and proceed along the same lines. The second supremum can be rewritten as $\sup_{s \in [0,1]} | \K_n^{(m)}(s)  + \Lb_n^{(m)}(s) |$, where, for any $s \in [0,1]$, 
\begin{align*}
\K_n^{(m)}(s) &= \frac{2}{\sqrt{n} (n-1)} \sum_{i,j=1 \atop i \neq j}^{\ip{ns}} \xi_{i,n}^{(m)}  h_2(\vec X_i, \vec X_j), \\ 
\Lb_n^{(m)}(s) &= \frac{2}{\sqrt{n} (n-1)} \sum_{i=1}^{\ip{ns}} \xi_{i,n}^{(m)}  \sum_{j=\ip{ns}+1}^n  h_2(\vec X_i, \vec X_j).
\end{align*}
The proof of the convergence in probability to zero of $\sup_{s \in [0,1]} | \K_n^{(m)}(s)|$ is very similar to that of Lemma~\ref{lem:Wn_p_0}. The same arguments can be adapted to show that $\sup_{s \in [0,1]} | \Lb_n^{(m)}(s)| = o_\Pr(1)$ by proving, in particular, appropriate versions of Lemmas~5.2 and 5.3 of \cite{DehFriGarWen15}.

Finally, from Proposition~\ref{prop:func_mult}, we have that
$$
\left( s \mapsto 2n^{-1/2} \sum_{i=1}^{\ip{ns}} h_1(\vec X_i), s \mapsto 2n^{-1/2} \sum_{i=1}^{\ip{ns}} \xi_{i,n}^{(1)} h_1(\vec X_i), \dots, s \mapsto 2n^{-1/2} \sum_{i=1}^{\ip{ns}} \xi_{i,n}^{(M)} h_1(\vec X_i), \right)
$$
converges weakly to $(\U,\U^{(1)},\dots,\U^{(M)})$ in $\{ \ell^\infty([0,1]) \}^2$. The last claim of the proposition then follows from the continuous mapping theorem and~\eqref{eq:asym_equiv_Un},~\eqref{eq:asym_equiv_hatUnm} and~\eqref{eq:asym_equiv_checkUnm}.
\end{proof}

\begin{lem}
\label{lem:Wn_p_0}
Assume that $\vec X_1,\dots,\vec X_n$ is drawn from a strictly stationary sequence $(\vec X_i)_{i \in \N}$ and that $h_2$ has uniform $(2+\delta)$-moments for some $\delta > 0$. Also, for any $m \in \{1,\dots,M\}$, let $(\xi_{i,n}^{(m)})_{i \in \N}$ be a sequence satisfying~(M1) in Definition~\ref{defn:mult_seq}. Furthermore, suppose that there exists $\tau \in [0,1)$ such that one of the following two conditions holds:
\begin{enumerate}[(i)]
\item $(\vec X_i)_{i \in \N}$ is absolutely regular and $\sum_{r=0}^n r \beta_r^{\delta/(2+\delta)} = O(n^\tau)$,
\item $(\vec X_i)_{i \in \N}$ is strongly mixing, $\Ex ( \| \vec X_1 \|^\gamma ) < \infty$ for some $\gamma > 0$, $h_2$ satisfies the $\Pr$-Lipschitz continuity or variation condition and $\sum_{r=0}^n r \alpha_r^{2\gamma\delta/(3\gamma\delta + \delta + 5\gamma + 2)} = O(n^\tau)$.
\end{enumerate}
Then, for any $m \in \{1,\dots,M\}$, $\sup_{s\in[2/n,1]} | \W_n^{(m)}(s) | \p 0$, where $\W_n^{(m)} $ is defined in~\eqref{eq:Wnm}.
\end{lem}

\begin{proof}
Fix $m \in \{1,\dots,M\}$. We shall first show that $\sup_{s\in[2/n,n^{-1/2}]} | \W_n^{(m)}(s) | = o_\Pr(1)$. Clearly,
$$
\sup_{s\in[2/n,n^{-1/2}]} | \W_n^{(m)}(s) | = \max_{2 \leq k \leq n^{1/2}} | \W_n^{(m)}(k/n) |.
$$
Furthermore,
$$
\Ex \left[ \left\{ \max_{2 \leq k \leq n^{1/2}} | \W_n^{(m)}(k/n) | \right\}^2 \right] = \Ex \left[ \max_{2 \leq k \leq n^{1/2}} \{ \W_n^{(m)}(k/n) \} ^2 \right] \leq \sum_{2 \leq k \leq n^{1/2}} \Ex [ \{ \W_n^{(m)}(k/n) \}^2 ].
$$
Using the fact that $\Ex(\xi_{i_1,n}^{(m)}\xi_{i_2,n}^{(m)}) \leq \Ex\{ (\xi_{i_1,n}^{(m)})^2 \} = 1$ and Lemma 4.4 in \cite{DehWen10b}, we obtain
\begin{align*}
\Ex [ \{ \W_n^{(m)}(k/n) \}^2 ] &= \frac{4}{n (k-1)^2} \sum_{i_1,j_1,i_2,j_2=1 \atop i_1 \neq j_1, i_2 \neq j_2}^k \Ex(\xi_{i_1,n}^{(m)}\xi_{i_2,n}^{(m)}) \Ex\{ h_2(\vec X_{i_1}, \vec X_{j_1})  h_2(\vec X_{i_2}, \vec X_{j_2}) \} \\
&\leq \frac{4}{n (k-1)^2} \sum_{i_1,j_1,i_2,j_2=1}^k | \Ex\{ h_2(\vec X_{i_1}, \vec X_{j_1})  h_2(\vec X_{i_2}, \vec X_{j_2}) \} | = O(k^\tau/n).
\end{align*}
It follows that 
\begin{equation}
\label{eq:first_max}
\Ex \left[ \left\{ \max_{2 \leq k \leq n^{1/2}} | \W_n^{(m)}(k/n) | \right\}^2 \right] = O(n^{(\tau-1)/2}) \to 0.
\end{equation}

It remains to show that $\sup_{s\in[n^{-1/2},1]} | \W_n^{(m)}(s) | = o_\Pr(1)$. To do so, we use Lemma~\ref{lem:increments_Wn} below. Using Markov's inequality, for any $n^{-1/2} \leq s < t \leq 1$ such that $t-s \geq n^{-1}$ and any $\lambda > 0$, we have
$$
\Pr\{ | \W_n^{(m)}(s) - \W_n^{(m)}(t) | \geq \lambda \} \leq \lambda^{-2} \Ex[ \{\W_n^{(m)}(s) - \W_n^{(m)}(t)\}^2 ] \leq \lambda^{-2} K (t-s) n^{(\tau-1)/2},
$$
which, as $\W_n^{(m)}(s) = \W_n^{(m)}(\ip{ns}/n)$ for all $s \in [2/n,1]$, implies that
\begin{align*}
\Pr\{ | \W_n^{(m)}(\ip{ns}/n) - \W_n^{(m)}(\ip{nt}/n) | \geq \lambda \} &\leq \lambda^{-2} K (nt-ns) n^{(\tau-3)/2} \\ &\leq \lambda^{-2} K 2(\ip{nt}-\ip{ns}) n^{(\tau-3)/2}.
\end{align*}
Setting $\kappa=2K$, it follows that, for any integers $\ip{n^{1/2}}\leq k \leq l \leq n$, 
\begin{equation}
\label{eq:tightness_eq}
\Pr\{ | \W_n^{(m)}(k/n) - \W_n^{(m)}(l/n) | \geq \lambda \} \leq \lambda^{-2} \kappa (l-k) n^{(\tau-3)/2} \leq\lambda^{-2} \kappa (l-k)^{1 + \epsilon} n^{(\tau-3)/2}
\end{equation}
for some $0 < \epsilon < (1-\tau)/2$. Then, for any $i \in \{1,\dots,n - \ip{n^{1/2}}\}$, set
$$
\zeta_i = \W_n^{(m)}\{ (i + \ip{n^{1/2}})/n \} - \W_n^{(m)}\{ (i - 1  + \ip{n^{1/2}})/n \}.
$$
Also, let $S_0 = 0$ and, for any $i \in \{1,\dots,n - \ip{n^{1/2}}\}$, let $S_i = \sum_{j=1}^i \zeta_j$. Thus, $S_i = \W_n^{(m)}\{ (i + \ip{n^{1/2}})/n \} - \W_n^{(m)}(\ip{n^{1/2}})/n)$ and, from~\eqref{eq:tightness_eq}, we have that, for any integers $0 \leq k \leq l \leq n - \ip{n^{1/2}}$,
$$
\Pr\{ | S_l - S_k | \geq \lambda \} \leq \lambda^{-2} [\kappa^{1/(1+\epsilon)} n^{(\tau-3)/\{2(1+\epsilon)\} } (l-k) ]^{1 + \epsilon}.
$$
We can then apply Theorem~10.2 of \cite{Bil99} with $\alpha=(1+\epsilon)/2$, $\beta=1/2$ and $u_l =  \kappa^{1/(1+\epsilon)} n^{(\tau-3)/\{ 2(1+\epsilon) \} }$, $l \in \{1,\dots,n - \ip{n^{1/2}}\}$. Hence, there exists a constant $\kappa' > 0$ depending only $\epsilon$ such that, for any $\lambda > 0$, 
\begin{align*}
\Pr \left( \max_{1 \leq k \leq n - \ip{n^{1/2}}} |S_k| \geq \lambda \right) &\leq \kappa' \lambda^{-2} \left\{ \kappa^{1/(1+\epsilon)} n^{(\tau-3)/\{ 2(1+\epsilon) \} } (n - \ip{n^{1/2}}) \right\}^{1+\epsilon} \\
&\leq \kappa' \lambda^{-2} \kappa n^{(\tau-3)/2 + 1 + \epsilon} \to 0.
\end{align*}
The desired result finally follows from~\eqref{eq:first_max} and the fact that
$$
\max_{\ip{n^{1/2}}+1 \leq k \leq n} |\W_n^{(m)}(k/n)| \leq \max_{1 \leq k \leq n - \ip{n^{1/2}}} |S_k| + |\W_n^{(m)}(\ip{n^{1/2}}/n)|.
$$
\end{proof}

\begin{lem}
\label{lem:increments_Wn}
Under the conditions of Lemma~\ref{lem:Wn_p_0}, there exists a constant $K >0$ such that, for any $m \in \{1,\dots,M\}$ and any $n^{-1/2} \leq s < t \leq 1$ such that $t-s \geq n^{-1}$,
$$
\Ex[ \{\W_n^{(m)}(s) - \W_n^{(m)}(t)\}^2 ] \leq K (t-s) n^{(\tau-1)/2}.
$$
\end{lem}

\begin{proof}
Fix $n^{-1/2} \leq s<t \leq 1$ such that $t-s \geq n^{-1}$. Then, 
\begin{multline*}
n^{1/2} \{ \W_n^{(m)}(t) - \W_n^{(m)}(s) \} / 2 = \frac{1}{\ip{nt} - 1} \sum_{i,j=1 \atop i \neq j}^{\ip{nt}}  \xi_{i,n}^{(m)} h_2(\vec X_i, \vec X_j) - \frac{1}{\ip{nt} - 1} \sum_{i,j=1 \atop i \neq j}^{\ip{ns}}  \xi_{i,n}^{(m)} h_2(\vec X_i, \vec X_j)  \\ + \frac{1}{\ip{nt} - 1} \sum_{i,j=1 \atop i \neq j}^{\ip{ns}}  \xi_{i,n}^{(m)} h_2(\vec X_i, \vec X_j) - \frac{1}{\ip{ns} - 1} \sum_{i,j=1 \atop i \neq j}^{\ip{ns}}  \xi_{i,n}^{(m)} h_2(\vec X_i, \vec X_j),
\end{multline*}
that is, $n^{1/2} \{ \W_n^{(m)}(t) - \W_n^{(m)}(s) \} / 2 = I_{n,1} + I_{n,1}' + I_{n,2} + I_{n,2}' + I_{n,3}$, where
\begin{align*}
I_{n,1} &= \frac{1}{\ip{nt} - 1} \sum_{\ip{ns}+1 \leq i < j \leq \ip{nt}}  \xi_{i,n}^{(m)} h_2(\vec X_i, \vec X_j), \\
I_{n,1}' &= \frac{1}{\ip{nt} - 1} \sum_{\ip{ns}+1 \leq i < j \leq \ip{nt}}  \xi_{j,n}^{(m)} h_2(\vec X_i, \vec X_j), \\
I_{n,2} &= \frac{1}{\ip{nt} - 1} \sum_{1 \leq i \leq \ip{ns}  < j \leq \ip{nt}}  \xi_{i,n}^{(m)} h_2(\vec X_i, \vec X_j), \\
I_{n,2}' &= \frac{1}{\ip{nt} - 1} \sum_{1 \leq i \leq \ip{ns}  < j \leq \ip{nt}}  \xi_{j,n}^{(m)} h_2(\vec X_i, \vec X_j), \\
I_{n,3} &= \frac{\ip{ns} - \ip{nt}}{(\ip{ns} - 1)(\ip{nt} - 1)} \sum_{i,j=1 \atop i \neq j}^{\ip{ns}}  \xi_{i,n}^{(m)} h_2(\vec X_i, \vec X_j).
\end{align*}
Using the fact that, for any $x,y \in \R$, $(x+y)^2 \leq 2 (x^2 + y^2)$, we have that
\begin{equation}
\label{eq:increments}
\Ex[ n \{ \W_n^{(m)}(t) - \W_n^{(m)}(s) \}^2 / 4 ] \leq 8 \{ \Ex(I_{n,1}^2) + \Ex(I_{n,1}'^2) + \Ex(I_{n,2}^2) + \Ex(I_{n,2}'^2) + \Ex(I_{n,3}^2)\}.
\end{equation}
Now, 
\begin{multline*}
\Ex(I_{n,1}^2) 
\leq 
\frac{1}{(\ip{nt} - 1)^2} \sum_{\ip{ns}+1 \leq i_1 < j_1 \leq \ip{nt} \atop \ip{ns}+1 \leq i_2 < j_2 \leq \ip{nt}} |\Ex\{ h_2(\vec X_{i_1}, \vec X_{j_1}) h_2(\vec X_{i_2}, \vec X_{j_2}) \} |  \\
= 
O \left(\frac{(\ip{nt} - \ip{ns})^{2+\tau}}{(\ip{nt} - 1)^2} \right),
\end{multline*}
where we have used the fact that $\Ex(\xi_{i_1,n}^{(m)}\xi_{i_2,n}^{(m)}) \leq \Ex\{ (\xi_{i_1,n}^{(m)})^2 \} = 1$, Lemma 4.4 in \cite{DehWen10b}, and the stationarity of $(\vec X_i)_{i \in \N}$. Then, using the fact that $\ip{nt} - \ip{ns} \leq \ip{nt} - 1$, $n^{1/2} \leq nt$ and $t-s \geq n^{-1}$ , we obtain that 
$$
\frac{(\ip{nt} - \ip{ns})^{2+\tau}}{(\ip{nt} - 1)^2} \leq \frac{\ip{nt} - \ip{ns}}{(\ip{nt} - 1)^{1 - \tau}} \leq \frac{nt - ns +1}{(n^{1/2} - 2)^{1-\tau}} \leq \frac{n \times 2(t-s)}{(n^{1/2} - 2)^{1-\tau}},
$$
and therefore that $\Ex(I_{n,1}^2) = (t-s) \times O( n^{(\tau+1)/2} )$. Similarly, $\Ex(I_{n,1}'^2) = (t-s) \times O( n^{(\tau+1)/2} )$. Concerning $I_{n,2}$, we have
$$
\Ex(I_{n,2}^2) \leq \frac{1}{(\ip{nt} - 1)^2} \sum_{1 \leq i_1 \leq \ip{ns}  < j_1 \leq \ip{nt} \atop 1 \leq i_2 \leq \ip{ns}  < j_2 \leq \ip{nt}} |\Ex\{ h_2(\vec X_{i_1}, \vec X_{j_1}) h_2(\vec X_{i_2}, \vec X_{j_2}) \} |.
$$
Proceeding as in Lemma 5.2 of \cite{DehFriGarWen15}, and with the help of Lemmas 4.1 and 4.2 in \cite{DehWen10b}, we obtain
$$
\Ex(I_{n,2}^2) = O \left(\frac{\ip{ns} (\ip{nt} - \ip{ns}) \ip{nt}^\tau}{(\ip{nt} - 1)^2} \right) = O \left(\frac{\ip{nt} - \ip{ns}}{(\ip{nt} - 1)^{1 - \tau}} \right) = (t-s) \times O( n^{(\tau+1)/2} ),
$$
and similarly for $I_{n,2}'$ and $I_{n,3}$. The desired result finally follows from~\eqref{eq:increments}.
\end{proof}

\section{Proofs of Propositions~\ref{prop:weakDn} and~\ref{prop:multDn}}
\label{proofs:Dn}

\begin{proof}[\bf Proof of Proposition~\ref{prop:weakDn}]
A first step consists of showing that 
$$
\sup_{s \in [0,1]} | \U_n^*(s) -  \frac{2}{\sqrt{n}} \sum_{i=\ip{ns}+1}^n h_1(\vec X_i)| = o_\Pr(1).
$$ 
To do so, it suffices to prove that $\sup_{s \in [0,1-2/n]} | \U_n^*(s) -  2n^{-1/2} \sum_{i=\ip{ns}+1}^n h_1(\vec X_i)| = o_\Pr(1)$. Using~\eqref{eq:Hoeff_decomp}, this amounts to showing that $n^{-1/2} \max_{2 \leq k \leq n} k |U_{h_2,n-k+1:n}| = o_\Pr(1)$. The latter can be proved by adapting the arguments used in Lemmas~\ref{lem:Wn_p_0} and~\ref{lem:increments_Wn}. 

Using the above result and~\eqref{eq:asym_equiv_Un}, we obtain that 
\begin{equation}
\label{eq:asym_equiv_Dn}
\sup_{s\in [0,1]} \left| \D_n(s) - (1-s) \times \frac{2}{\sqrt{n}} \sum_{i=1}^{\ip{ns}} h_1(\vec X_i) + s \times \frac{2}{\sqrt{n}} \sum_{\ip{ns}+1}^n h_1(\vec X_i)\right| = o_\Pr(1),
\end{equation}
and the desired result follows from Theorem~2 of \cite{OodYos72} and the continuous mapping theorem.
\end{proof}

\begin{proof}[\bf Proof of Proposition~\ref{prop:multDn}] 
For any $m \in \{1,\dots,M\}$, let
$$
\tilde \U_n^{*,(m)} (s)  = \frac{2}{\sqrt{n}} \sum_{i=\ip{ns}+1}^{n} \xi_{i,n}^{(m)} \tilde h_{1,\ip{ns}+1:n} (\vec X_i), \qquad s \in [0,1],
$$
where $\tilde h_{1,k:l}(X_i)$ is defined in \eqref{eq:tildeh1kl}. Fix $m \in \{1,\dots,M\}$.  Adapting the arguments used in the proof of Proposition~\ref{prop:multUn}, it can be verified that $\check \U_n^{*,(m)}$, defined in \eqref{eq:checkUn*m}, is asymptotically equivalent to $\tilde \U_n^{*,(m)}$, and that $\sup_{s \in [0,1]} | \tilde \U_n^{*,(m)} (s) -  2n^{-1/2} \sum_{i=\ip{ns}+1}^n \xi_{i,n}^{(m)} h_1(\vec X_i)| = o_\Pr(1)$. Combined with~\eqref{eq:asym_equiv_tildeUnm}, we then obtain that
$$
\sup_{s\in [0,1]} \left| \check \D_n^{(m)}(s) - (1-s) \times \frac{2}{\sqrt{n}} \sum_{i=1}^{\ip{ns}} \xi_{i,n}^{(m)} h_1(\vec X_i) + s \times \frac{2}{\sqrt{n}} \sum_{\ip{ns}+1}^n \xi_{i,n}^{(m)} h_1(\vec X_i)\right| = o_\Pr(1),
$$
where $\check \D_n^{(m)}$ is defined in~\eqref{eq:checkDnm}. 
Similar arguments can be carried out to get the analogue display for $\hat \D_n^{(m)}$ in~\eqref{eq:hatDnm}. The desired result finally additionally follows from Proposition~\ref{prop:func_mult}, the continuous mapping theorem and~\eqref{eq:asym_equiv_Dn}.
\end{proof}

\section*{Acknowledgements}
This work has been supported in parts by the Collaborative Research Center ``Statistical modeling of nonlinear dynamic processes'' (SFB 823) of the German Research Foundation (DFG), which is gratefully acknowledged.

\bibliographystyle{plainnat}
\bibliography{biblio}

\end{document}